\documentclass{amsart}

\usepackage{hyperref}

\usepackage{pdfrender,xcolor}

\hypersetup{colorlinks=true,
	linkcolor=black, 
	urlcolor=blue} 
\usepackage{amsmath,amssymb,amsfonts,amsthm}
\usepackage{amsthm}
\usepackage{enumerate}
\numberwithin{equation}{section}
\usepackage{marginnote}
\usepackage[a4paper]{geometry}

\reversemarginpar

\usepackage[english]{babel}
\usepackage[babel, english=american]{csquotes}

\newtheoremstyle{Teorema}
{3pt}
{3pt}
{\slshape}
{}
{\bfseries}
{:}
{\newline}
{}

\newtheorem{theorem}{Theorem}[section]
\newtheorem{definition}[theorem]{Definition}

\newtheorem{corollary}[theorem]{Corollary}

\newtheorem{proposition}[theorem]{Proposition}

\newtheorem{lemma}[theorem]{Lemma}

\theoremstyle{definition}

\DeclareMathOperator{\D}{d}
\DeclareMathOperator{\supp}{supp}

\newcommand{\Z}{\mathbb{Z}}
\newcommand{\R}{\mathbb{R}}
\newcommand{\C}{\mathbb{C}}

\date{\today}

\begin{document}

\pdfrender{StrokeColor=black,TextRenderingMode=2,LineWidth=0.4pt}

\title[Strong asymptotic of polynomials]{Strong asymptotic of Cauchy biorthogonal polynomials and orthogonal polynomials
with varying measure}

\author{L. G. Gonz\'alez Ricardo}
\address{Department of Mathematics, Universidad Carlos \textsc{iii} de Madrid, Avenida de la Universidad, 30
CP-28911, Leganés, Madrid, Spain.}
\email{luisggon@math.uc3m.es}
\thanks{The first author was supported in part by a research fellowship from Department of Mathematics of Universidad
Carlos III de Madrid, Spain.}

\author{G. L\'opez Lagomasino}
\address{Department of Mathematics, Universidad Carlos \textsc{iii} de Madrid, Avenida de la Universidad, 30
CP-28911, Leganés, Madrid, Spain.}
\email{lago@math.uc3m.es}
\thanks{The second author was supported in part by the research grant PGC2018-096504-B-C33 of Ministerio de
Ciencia, Innovaci\'on y Universidades, Spain.}

\date{\today}

\begin{abstract}
	We give the strong asymptotic  of Cauchy biorthogonal polynomials under the assumption that the defining measures are supported on bounded non intersecting intervals of the real line and satisfy Szeg\H{o}'s condition. The biorthogonal polynomials are connected with certain mixed type Hermite-Pad\'e polynomials, which verify full orthogonality relations with respect to certain varying measures. Thus, the strong asymptotic of orthogonal polynomials with respect to varying measures plays a key role in the study.
\end{abstract}

\maketitle

\medskip


\medskip

	
\section{Introduction}

\subsection{Cauchy biorthogonal polynomials}
	Let $\mathbf{\Delta} = (\Delta_1,\Delta_2)$ be a pair of intervals, contained in the real line $\R$, which have at most one common point. By $\mathcal{M}(\mathbf{\Delta})$ we denote the cone of all pairs $(\sigma_1,\sigma_2)$ of positive Borel measures with finite moments whose supports verify $\supp\sigma_k\subset \Delta_k$ and
	\[ \int \int \frac{\D \sigma_1(x)  \D \sigma_2(y)}{|x-y|} < \infty.
	\]

	Fix $(\sigma_1,\sigma_2) \in \mathcal{M}(\mathbf{\Delta})$. For each pair of non negative integers $(m,n) \in \Z^2_{\geq 0}$ there exists a pair $(P_m, Q_n)$ of monic polynomials whose degrees verify $\deg P_m \leq m, \deg Q_n \leq n,$ and
	\begin{equation}
		\label{biort}
	\int_{\Delta_1} \int_{\Delta_2} \frac{P_m(x)  Q_n(y) \D \sigma_1(x) \D \sigma_2(y)}{x-y}  = C_n \delta_{m,n},\quad C_n \neq 0.
	\end{equation}
	(As usual, $\delta_{m,n} = 0, m \neq n, \delta_{n,n} = 1$.) These  polynomials were introduced in \cite{Bertola:CBOPs} and called Cauchy biorthogonal polynomials. The original definition uses the kernel $(x+y)^{-1}$ (and measures supported in the positive real line to avoid singularities in the kernel except when $x=y=0$), but we find it more convenient to employ $(x-y)^{-1}$ instead, since it adapts better to our presentation. Some interesting properties were revealed. In particular, it was shown that $\deg P_n = n$, its zeros are simple, interlace for consecutive values of $n$, and lie in $\stackrel{\circ}{\Delta}_1$ (the interior of $\Delta_1$ with the Euclidean topology of $\R$). The same goes for the $Q_n$ on $\Delta_2$.

	Cauchy biorthogonal polynomials appear in the analysis of the two matrix model \cite{Ber,BGS} and were used to find discrete solutions of the Degasperis-Procesi equation \cite{Bertola:CBOPs} through a Hermite-Pad\'e approximation problem for two discrete measures. In \cite{BGS}, the authors apply the nonlinear steepest descent method to a class of $3 \times 3$ Riemann-Hilbert problems introduced in connection with the Cauchy two-matrix random model, solve the Riemann-Hilbert problem, and establish strong asymptotic results for the Cauchy biorthogonal polynomials for a class of measures given by weights with exponential decay at infinity (of Laguerre type).  The results obtained in \cite{BGS} were later extended in \cite{BerBoth}.

	Our goal is to prove strong asymptotic results for Cauchy biorthogonal polynomials when  the intervals
	$\Delta_k=[a_k,b_k]$, $k=1,2$ are bounded non intersecting intervals and the measures $\sigma_1, \sigma_2$  verify Szeg\H{o}'s condition
	\begin{equation}
		\label{szego}
		\int_{\Delta_k} \ln \sigma_k'(x)\D\eta_{\Delta_k}(x)  > - \infty, \qquad k=1,2,
	\end{equation}
	where $\sigma'$ denotes the Radon-Nikodym derivative of $\sigma$ with respect to the Lebesgue measure  and
	$$\D\eta_{\Delta}(x) := \frac{\D x}{\sqrt{(x-a)(b-x)}}$$
	stands for Chebyshev measure in the interval $\Delta = [a,b]$. In this case we write $(\sigma_1,\sigma_2) \in \mathcal{S}(\mathbf{\Delta})$. Therefore, we extend Szeg\H{o}'s theory on the strong asymptotic of orthogonal polynomials supported on a bounded interval of the real line to the context of Cauchy biorthogonal polynomials. In the sequel, the intervals $\Delta_1,\Delta_2$ are bounded and do not intersect.

	\begin{theorem}
		\label{mainbiort}
		Let $(\sigma_1,\sigma_2) \in \mathcal{S}(\mathbf{\Delta})$ and $(P_n)_{n\geq 0}, (Q_n)_{n\geq 0}$ be the sequences of monic polynomials determined by \eqref{biort}. Then
		\begin{equation}
			\label{asinbio}
			\lim_{n} \frac{P_n(z)}{\Phi_1^n(z)} = \frac{\mathsf{G}^*_1(z)}{\mathsf{G}^*_1(\infty)}, \qquad \lim_{n} \frac{Q_n(z)}{\Phi_2^n(z)} = \frac{\mathsf{G}_2(z)}{\mathsf{G}_2(\infty)},
		\end{equation}
		uniformly on each compact subset of $\Omega_1 = \overline{\C} \setminus \Delta_1$ and $\Omega_2 =\overline{\C} \setminus \Delta_2$, respectively, where $\Phi_k \in \mathcal{H}(\Omega_k), k=1,2,$ (holomorphic in $\Omega_k $) is the exponential of a complex potential constructed from a vector equilibrium problem (see \eqref{vector_equil} and \eqref{compfunc}), and $\mathsf{G}^*_1, \mathsf{G}_2$ are Szeg\H{o} functions obtained as components  of fixed points of the maps $T_{\mathbf{w}_P}$ and $T_{\mathbf{w}_Q}$, respectively (see Def. \ref{oper_T}, \eqref{oper_TP} and \eqref{oper_TQ}).
	\end{theorem}

The logarithmic and ratio asymptotic of biorthogonal polynomials were obtained in \cite{ULS} for more general Cauchy type kernels involving $m\geq 2$ measures. As in \cite{ULS}, we reduce the study of the strong asymptotic of Cauchy biorthogonal polynomials to that of  polynomials arising from an associated mixed type Hermite-Pad\'e approximation problem. The Hermite-Pad\'e polynomials turn out to be  orthogonal  with respect to varying measures. So, the strong asymptotic of such sequences of orthogonal polynomials  play a central role in our discussion.

\subsection{Orthogonal polynomials with varying measures}
	Let $\Delta = [a,b] \subset \R$. Consider a sequence $(\D \mu_n/w_{2n})_{n\geq 0}$ where $\mu_n$ is a finite positive Borel measure supported on $\Delta$ and $w_{2n}$ is a polynomial with real coefficients, $\deg w_{2n}=i_n\leq 2n$, whose zeros $(x_{2n,i})_{i=2n-i_n+1}^{2n}$ lie in $\C \setminus \Delta$. This is called a sequence of varying measures. Let $L_n(x) = x^n+\cdots$ be the $n$-th monic orthogonal polynomial satisfying
	\begin{equation}
		\label{def:Ln}
		\int x^\nu L_n(x)\frac{\D \mu_n(x)}{|w_{2n}(x)|} =0,\qquad \nu=0,1,\ldots,n-1.
	\end{equation}
	The sequence $(L_n)_{n \geq 0}$ is called the sequence of monic orthogonal polynomials with respect to the given varying measures. A common normalization is to take
	\[ \tau_n := \left(\int  L_n^2(x)\frac{\D \mu_n(x)}{|w_{2n}(x)|}\right)^{-1/2},
  	\]
	and define $l_n(x):=\tau_n L_n(x)$ as the orthonormal polynomial of degree $n$.

	In the context of multipoint Pad\'e and Hermite-Pad\'e approximation, orthogonal polynomials with respect to varying measures arise naturally (see, for example, \cite{sasha, BSW, buslop, GRS, lago89}). Depending on the type of asymptotic one wishes to obtain for the sequence $(L_n)_{n\geq0}$ (or $(l_n)_{n\geq0}$), some conditions must be imposed on the varying measures. Combinations of the following ones are appropriate in the present paper.
	\begin{itemize}
		\item[i)] There exists a finite positive Borel measure $\mu$ supported on $\Delta$ such that $\lim_{n}	\mu_n = \mu$ in the weak star topology of measures, whose absolutely continuous part satisfies   $\mu' > 0$ a.e. on $\Delta$, and
			\[ \lim_{n} \int |\mu_n' - \mu'| \D x = 0.
			\]
		\item[ii)] The measure $\mu$ satisfies Szeg\H{o}'s condition on $\Delta$; that is,
			\[    \int_\Delta \ln \mu'(x)\D\eta_\Delta(x)>-\infty
			\]
			and
			\[
			\liminf_n \int \ln \mu_n'(x) \D\eta_\Delta(x) \geq \int \ln \mu'(x) \D\eta_\Delta(x).\]
		\item[iii)] Let $\Psi$ be the conformal map from $\Omega = \overline{\C}\setminus\Delta$ onto the 	  exterior of the unit circle such that $\Psi(\infty)=\infty$ and $\Psi'(\infty)>0$. The zeros 		of the polynomials $w_{2n}$ verify
			\[ \lim_{n\to \infty} \sum_{i=  1}^{2n}\left(1 - \frac{1}{|\Psi(x_{2n,i})|} \right) = \infty.
			\]
			By convention, $x_{2n,i} = \infty, 1\leq i \leq 2n-i_n,$ when $i_n < 2n$.
		\item[iv)] There exist non negative  continuous functions $\varphi$ and $\psi$ on $(a,b)$ such that
			\begin{equation}
				\label{def:psi}
				\lim_n \varphi^{n}(x)|w_{2n}(x)| = 1/\psi(x)
			\end{equation}
			uniformly on compact subsets of $(a,b)$ and
			\begin{equation}
   				\label{rest_var_w}
   				\lim_{n \to \infty}\int_a^b \ln(\varphi^{n}(x)|w_{2n}(x)|)\D\eta_\Delta(x)= -\int_a^b \ln
   \psi(x)\D\eta_\Delta(x) < +\infty.
 			\end{equation}
	\end{itemize}

	In many applications, $\D \mu_n = h_n \D \overline{\mu}, \overline{\mu}' > 0$ a.e. on $\Delta$, where $(h_n)_{n \geq 0}$ is a sequence of positive continuous functions which converges uniformly on $\Delta$ to a positive continuous function $h$, and the zeros of the polynomials $(w_{2n})_{n\geq 0}$ are uniformly bounded away from $\Delta$ in which case $i)$ and $iii)$ are immediate, and $ii)$ holds if $\overline\mu$ verifies Szeg\H{o}'s condition.

	Conditions $i)-iii)$ are sufficient to prove strong asymptotic for $(L_n)_{n\geq 0}$. The first result in this direction appeared in \cite{lago} and was later improved in \cite{bernardo_lago2} and \cite{BCL}. An alternative proof of the main result in \cite{lago} may be found in \cite{stahl}.  The answer in \cite{lago} is given in terms of a Szeg\H{o} function associated with $\mu$ and a Blaschke product in which the zeros of $w_{2n}$ intervene (see \eqref{asintb} below).  We wish to replace the Blaschke product in the asymptotic formula by the $n$-th power of a fixed function (as in Szeg\H{o}'s classical result). In order to achieve this, some knowledge of the asymptotic behavior of the polynomials $w_{2n}$ is required and condition $iv)$ comes in.

	Let $\varphi$ be a positive continuous function on $\Delta$. Let $\lambda_\varphi$ be the (unitary) equilibrium measure supported on $\Delta$ which solves the equilibrium problem for the logarithmic potential with external field $-\frac{1}{2}\ln \varphi$. It is well known that $\lambda_{\varphi}$ is uniquely determined by the equilibrium conditions on $\Delta$ (see \cite[Theorem \textsc{i}.3]{saff_totik})
	\begin{equation}
		\label{equilibrium}
  		V_{\lambda_\varphi}(x)-\frac{1}{2}\ln \varphi(x)
  		\begin{cases}
  			\leq \gamma, & x\in\supp\lambda_\varphi,\\
  			\geq \gamma, & x\in\Delta\setminus e,\,\, \mbox{cap}(e) = 0,
  		\end{cases}
	\end{equation}
	where $\gamma$ is a constant, $\mbox{cap}(e)$ denotes the logarithmic capacity of $e$, and
	\[ V_{\lambda_\varphi}(u) = \int \ln \frac{1}{|u-x|} \D \lambda_\varphi(x)
	\]
	denotes the logarithmic potential of $\lambda_\varphi$. We will assume that $\varphi$ is such that $\supp \lambda_{\varphi} = \Delta$. (This is true, for example, if $\frac{1}{2}\ln \varphi= V_\rho$ is the logarithmic potential of a measure $\rho$, of total mass $c \leq 1$, supported on an interval disjoint from $\Delta$. In this case, $\lambda_{\varphi}$ is the balayage of $\rho$ on $\Delta$ plus $(1-c)$ times $ \D \eta_\Delta/\pi$, and equality is attained in \eqref{equilibrium} on all  $\Delta$ due to the regularity of the interval $\Delta$ with respect to the Dirichlet problem.)

	Set
	\begin{equation}
		\label{PhiC}
  		\Phi(u):= e^{-v_\varphi(u)},\qquad v_\varphi := V_{\lambda_\varphi} + i\widetilde{V}_{\lambda_\varphi},\qquad  C:=e^{\gamma},
	\end{equation}
	where $\widetilde{V}_{\lambda_\varphi}$ denotes the harmonic conjugate of $ V_{\lambda_\varphi}$ in $ {\C} \setminus \Delta$ (which equals zero when $u > b, \Delta = [a,b]$). Though $\widetilde{V}_{\lambda_\varphi}$ is multi-valued, it has an increment of $2\pi$ if we surround once the interval $\Delta$ in the positive direction; consequently, $\Phi$ is a single-valued analytic function in $\C \setminus \Delta$ with a simple pole at $\infty$ since $\Phi(u) = u + \mathcal{O}(1), u \to \infty$.

	We write $\mu \in \mathcal{S}(\Delta)$ when $\mu$ verifies Szeg\H{o}'s condition on $\Delta$. Then, the Szeg\H{o} function of $\mu$ is defined as
	\begin{equation}
		\label{green}
		\mathsf{G}(\mu,u) := \exp\left[\frac{\sqrt{(u-b)(u-a)}}{2\pi}\int_\Delta\frac{\ln(\sqrt{(b-x)(x-a)}\,\mu'(x))}{x-u}\D\eta_\Delta(x)\right].
	\end{equation}
	The square root outside the integral is taken to be positive for $u > b$ and those inside the integral are positive when $x \in (a,b)$. The Szeg\H{o} function is characterized in terms of a boundary value problem (for details see Section \ref{Szego:fun}).

	\begin{theorem}
 		\label{main_th}
 		Assume that $(\mu_n,w_{2n})_{n \geq 0}$ verifies $i)-iv)$ and $\supp \lambda_\varphi = \Delta$ (see
 \eqref{equilibrium}). Then,
 			\begin{equation}
 			\label{strong}
 			\lim_n \frac{l_n(u)}{C^{n}\Phi^{n}(u)} = \frac{1}{\sqrt{2\pi}}\mathsf{G}(\psi   \mu,u),
 		\end{equation}
 		uniformly on compact subsets of $\Omega$, $\psi \mu$ is the measure with differential expression $\psi\D\mu$,
 and $\psi$ is given by \eqref{def:psi}. Moreover,
	 	\begin{equation}
			\label{constante}
 			\lim_n \frac{\tau_n}{C^{n}} = \frac{1}{\sqrt{2\pi}}\mathsf{G}(\psi  \mu, \infty),
 		\end{equation}
 		and
	 	\begin{equation}
			\label{asintc}
 			\lim_n \frac{L_n(u)}{\Phi^{n}(u)} = \frac{\mathsf{G}(\psi  \mu,u)}{\mathsf{G}(\psi  \mu,\infty)}
 		\end{equation}
 		uniformly on compact subsets of $\Omega$.
 	\end{theorem}

 The following result is obtained from Theorem \ref{main_th}. It is in the spirit of \cite[Theorem 14.3]{Totik}. The  assumptions  have points in common but they are not the same. In some regards the conditions in \cite{Totik} are more general, in others our assumptions are weaker. The most notable difference  is that in \cite[Theorem 14.3]{Totik} the measure $\mu$ is required to be absolutely continuous with respect to the Lebesgue measure whereas we do not need this restriction.

	\begin{theorem}
		\label{cor:a}
		Let $(\mu_n)_{n \geq 0}$ be a sequence of measures verifying $i)-ii)$. Let $\tau, \D \tau = v \D x,$ be a probability
measure on $\Delta$ such that $\supp \tau = \Delta,  v $ is continuous on $\Delta,$ and let there be constants $A, \beta >
-1,$ and $\beta_0$ such that
		\begin{equation}
			\label{cond:a}
			A^{-1}((b-x)(x-a))^{\beta_0}\leq v(x) \leq A((b-x)(x-a))^\beta, \qquad x \in (a,b).
		\end{equation}
		Set
		\[ \Phi_\tau(u) = \exp{(-V_{\tau}(u) - i\widetilde{V}_{\tau}(u))},
		\]
		where $\widetilde{V}_{\tau}$ is the harmonic conjugate in $\C \setminus \Delta$ of the logarithmic potential
${V_{\tau}} $. Then
		\begin{equation}
 			\label{strong:a}
 			\lim_n \frac{p_n(u)}{ \Phi_\tau^{n}(u)} = \frac{1}{\sqrt{2\pi}}\mathsf{G}(\mu,u),
 		\end{equation}
 		uniformly on compact subsets of $\Omega$, where $p_n$ is the $n$-th orthonormal polynomial verifying
 		\[ \int p_m(x) p_n(x) \frac{\D \mu_n(x)}{|\Phi_\tau^{2n}(x)|} = \left\{
 		\begin{array}{ll}
 			0, & m < n, \\
 			1, & m = n.
 		\end{array}\right.
 		\]
	\end{theorem}

\subsection{Outline and structure of the paper}

Section 2 is dedicated to the proof of Theorems \ref{main_th} and \ref{cor:a}. These results are used in Section 3 in the proof of Theorem \ref{mainbiort} but they have independent interest and may be employed to obtain exact estimates of the rate of convergence of  multipoint Pad\'e and Hermite-Pad\'e approximations.

Section 3 is devoted to the study of the strong asymptotic of a sequence of Hermite-Pad\'e polynomials intimately connected with the Cauchy biorthogonal polynomials defined above. The proof of Theorem \ref{mainbiort} is not simple because it requires several steps  some of which are quite technical. A brief description of the idea of the proof is helpful for a better understanding of the whole paper.

In \cite{LMS} the authors noticed the connection between Cauchy biorthogonal polynomials and a  so called
multilevel  Hermite-Padé approximation problem. For convenience of the reader,
we summarize this relationship in subsection 3.1. In subsections 3.2 and 3.3 we prove some useful formulas verified by these approximants and their associated polynomials. In particular,
the biorthogonal polynomials $(P_n,Q_n)_{n\geq 0}$ are identified with certain   Hermite-Pad\'e polynomials which turn out to be orthogonal  with respect to varying measures. So the initial problem is reduced to finding the strong asymptotic of the associated Hermite-Pad\'e polynomials.

The results in \cite{LMS}  clearly indicate which functions $\Phi_1,\Phi_2$ must be taken to compare the Hermite-Pad\'e polynomials to establish their strong asymptotic behavior.  This is explained in detail in subsection 3.4. These functions are the exponentials of the complex potentials associated with the equilibrium measures of the vector  equilibrium problem used to describe the logarithmic asymptotic of the same polynomials.

Because of the definition of biorthogonality, if the role of the measures $\sigma_1,\sigma_2$ is interchanged then the polynomials $P_n,Q_n$ are also interchanged; therefore, if the strong asymptotic of the sequence $(Q_n)_{n\geq 0}$ is obtained then that of the sequence $(P_n)_{n\geq 0}$ readily follows. Thus, we focus on $(Q_n)_{n\geq 0}$; more precisely, on their associated Hermite-Pad\'e polynomials.

To obtain their strong asymptotic we adapt a very clever method devised by A.I. Aptekarev to obtain the strong asymptotic of type II Hermite-Pad\'e polynomials for Angelesco \cite{sasha2} and Nikishin \cite{sasha} systems of measures using fixed point theorems. To undestand what this is about,  we need to advance some formulas.

We show that for each $n \geq 1$ there exist polynomials $Q_{n,1}, Q_{n,2}$, with $Q_{n,2} = Q_n$, and a continuous function $h_{n,1}$ on $\Delta_1$ such that
\[
0 = \int x^{\nu}  Q_{n,2}(x)  \frac{{\rm d}\sigma_{2}(x)}{|Q_{n,1}(x)|} , \qquad
0 = \int x^{\nu}  Q_{n,1}(x)    \frac{|h_{n,1}(x)|{\rm d}\sigma_{1}(x)}{|Q_{n,2}(x)|}, \qquad \nu = 0,\ldots, n-1
\]
where $\lim_n |h_{n,1}(x)| =  (\sqrt{x-a_2)(b_2 -x)})^{-1}$ uniformly on $\Delta_1$.  Theorem \ref{main_th} cannot be used directly because the polynomial which in one relation is orthogonal in the other relation appears in the denominator of the varying part of the measure. To handle this, we (temporarily) unlink this inter dependence.  For that purpose, for each $n \geq 0$ a (non-linear) operator $\widetilde{T}_n$ is introduced, defined on the set of all pairs $(\hat{Q}_1,\hat{Q}_2)$ of monic polynomials with real coefficients of degree $n$ with zeros in the complement of $\Delta_2$ and $\Delta_1$, respectively, such that $\tilde{T}_n(\hat{Q}_1,\hat{Q}_2) = (Q_1^*,Q_2^*)$ verifies
\begin{equation} \label{relort}
0 = \int x^{\nu}  Q^*_{2}(x)  \frac{{\rm d}\sigma_{2}(x)}{|\hat{Q}_{1}(x)|} , \qquad
0 = \int x^{\nu}  Q^*_{1}(x)    \frac{|h_{n,1}(x)|{\rm d}\sigma_{1}(x)}{|\hat{Q}_{2}(x)|}, \qquad \nu = 0,\ldots, n-1.
\end{equation}
Notice that $(Q_{n,1},Q_{n,2})$ is a fixed point of $\tilde{T}_n$. (Indeed, in subsection 3.6 a more general operator is defined where it is only required that the sequence $(|h_{n,1}|)_{n\geq 0}$ converges uniformly to a positive continuous on $\Delta_1$. This extension allows to cover other possible applications we have in mind.)

Take sequences of denominators $(\hat{Q}_{n,1})_{n\geq 1}, (\hat{Q}_{n,2})_{n\geq 1},$ and their associated by \eqref{relort} orthogonal polynomials $({Q}^*_{n,1})_{n\geq 0}, (Q^*_{n,2})_{n\geq 1}$. If we suppose that $\hat{g}_1,\hat{g}_2$ are the uniform limits on $\Delta_2$ and $\Delta_1$, respectively, of the sequences $(\hat{Q}_{n,1}/\Phi_1^n)_{n\geq 0}, (\hat{Q}_{n.2}/\Phi_2^n)_{n\geq 0}$,  using Theorem
\ref{main_th} in subsection 3.6 we obtain the strong asymptotic $(g^*_1,g^*_2)$ of $({Q}^*_{n,1}/\Phi_1^n)_{n\geq 0}, ({Q}^*_{n.2}/\Phi_2^n)_{n\geq 0}$. Previously, in subsection 3.5 using Theorem \ref{cor:a} we show that any pair $(\hat{g}_1,\hat{g}_2)$ of Szeg\H{o} functions on $\C \setminus \Delta_1, \C \setminus \Delta_2$, respectively, can be obtained as strong limits of the initial pair of sequences. From the boundary properties verified by Szeg\H{o} functions it turns out that  $(\hat{g}_1,\hat{g}_2)$ and $( {g}^*_1,{g}^*_2)$ are connected by the boundary value equations
		\begin{align*}
			|g_1^*(x)|^2 =& \frac{\hat{g}_2(x)\sqrt{(b_2 - x)(x-a_2)}}{\sqrt{(b_1 - x)(x-a_1)}\sigma'_1(x)}, \quad
			\mbox{a.e. on} \quad [a_1,b_1] = \Delta_1,\\ 
			|g_2^*(x)|^2 =& \frac{\hat{g}_1(x)}{\sqrt{(b_2 - x)(x-a_2)}\sigma'_2(x)}, \quad
			\mbox{a.e. on} \quad [a_2,b_2] = \Delta_2 
		\end{align*}	
where $(\sigma_1,\sigma_2)\in S(\mathbf{\Delta})$. This motivates the introduction of another operator in subsection 3.7 $T(\hat{g}_1,\hat{g}_2)= ({g}^*_1,\hat{g}^*_2)$ which on an appropriate metric space of functions is contractive and due to Banach's fixed point theorem it has a unique fixed point. (Indeed in subsection 3.7 a more general situation is considered but we limit ourselves here to the operator which is relevant in the case of multi level Hermite-Pad\'e polynomials.) The final step consists in showing that any neighborhood of the fixed point of the operator $T$ contains fixed points of the operators $\tilde{T}_n$ for all sufficiently large $n$. This is done in Theorem \ref{main} of subsection 3.8 using Brouwer's fixed point theorem. Theorem \ref{ML} is a simple corollary of Theorem \ref{main} applied to  multi level Hermite Pad\'e polynomials.

In subsection 3.10 we return to the biorthogonal polynomials. Since $Q_n = Q_{n,2}$, Theorem \ref{ML} gives directly the asymptotic of the $Q_n$. Then, we briefly discuss what needs to be done for the polynomials $P_n$. In subsection 3.10 we derive the strong asymptotic of other functions related with the multi level Hermite Pad\'e approximation problem and the final section 3.11 contains a different approach for defining the comparison functions $\Phi_1,\Phi_2$ on the basis of a three sheeted Riemann surface of genus zero.

\section{Strong asymptotic of orthogonal polynomials with varying measures}

	As mentioned above our goal here is to prove Theorems \ref{main_th} and \ref{cor:a}. They are essential in the proof of Theorem \ref{mainbiort}, but have independent interest and may find other applications. We begin explainig our choice of Szeg\H{o} function for measures supported on an interval of the real line.

\subsection{The Szeg\H{o}\label{Szego:fun} function}
	Let $ \mu \in \mathcal{S}([-1,1])$; that is
	\[ \int_{-1}^1 \frac{\ln \mu'(x) \D x}{\sqrt{1-x^2}} > - \infty.
	\]
	On the unit circle $\mathbb{T}$, one can define a symmetric measure $\sigma$ with the property that $\sigma(B) = \mu(B^*)$ whenever $B$ is a Borel set contained either in the upper or lower half of the unit circle and $B^*$ is its orthogonal projection on $[-1,1]$. It readily follows that
	\[ \sigma'(e^{it}) = |\sin t|\,\mu'(\cos t), \qquad t \in [0,2\pi]
	\]
	where $\sigma'$ and $\mu'$ denote the Radon-Nikodym derivatives of $\sigma$ and $\mu$ with respect to the Lebesgue measure on $\mathbb{T}$ and $[-1,1]$, respectively. If $\zeta = e^{it}$ and $x = \mbox{Re} (\zeta) = \cos t$, we can also write
	\[ \sigma'(\zeta) = \sqrt{1- x^2}\,\mu'(x), \qquad \zeta \in \mathbb{T}, \qquad x = \mbox{Re}(\zeta).
	\]
	Let
	$$\mathsf{S}(\sigma,z) = \exp\left[\frac{1}{4\pi}\int_\mathbb{T} \frac{\zeta + z}{\zeta -z}{\ln \sigma'(\zeta) } |\D\zeta|
\right],$$
	be the (standard) Szeg\H{o} function associated with the measure $\sigma$. Notice that if $\mu$ satisfies the Szeg\H{o} condition on $[-1,1]$ then $\int_{\mathbb{T}} \ln \sigma'(\zeta) |\D \zeta| > -\infty$; that is, $\sigma$ verifies Szeg\H{o}'s condition on $\mathbb{T}$.

	In general, when $\supp \mu = \Delta =[a,b]$ (not necessarily $[-1,1]$), we define $\sigma$ as it was done before out of the measure $\widetilde{\mu}$ supported on $[-1,1]$ such that $\widetilde{\mu}(B) = \mu(\{x\in [a,b]: \frac{2}{b-a}(x - \frac{b+a}{2})\in B\})$, for every Borel set $B \subset [-1,1]$. In this case
	\[ \sigma'(e^{it}) = \sqrt{(b-x)(x-a)} \mu'(x), \qquad x = \frac{b-a}{2} \cos t + \frac{b+a}{2}.
	\]

	We wish to define a Szeg\H{o} function $\mathsf{G}(\mu,\cdot)$ with respect to the measure $\mu$ so that 	
	\[ \mathsf{G}(\mu,u) = \mathsf{S}(\sigma, \Psi(u)), \qquad u \in \overline{\C} \setminus \Delta.
	\]
	Then, from known properties of the Szeg\H{o} function for measures on the unit circle, we have
	\begin{equation}
		\label{Glimit}
		\lim_{u \to x} |\mathsf{G}(\mu, u)|^2 = \lim_{u\to x}|\mathsf{S}(\sigma,\Psi(u))|^2 = 1/\sigma'(\zeta)
	 = (\sqrt{(b-x)(x-a)}\,\mu'(x))^{-1},\quad \mbox{a.e. on}\,\, \Delta,
	\end{equation}
	where $ \zeta = \Psi(x)$ ($\Psi$ can be extended continuously to $\Delta$ as usual assuming that the interval has two sides and since $\sigma$ is symmetric with respect to the real line we can take $\zeta$ either on the upper half or the lower half of $\mathbb{T}$). Straightforward calculations show that the explicit expression of $\mathsf{G}(\mu, u)$ is \eqref{green}.

	The property stated in \eqref{Glimit} also serves to characterize the Szeg\H{o} function associated with a measure $\mu \in \mathcal{S}(\Delta)$ by a boundary value problem. This is, given $\mu \in \mathcal{S}(\Delta)$ find a holomorphic function $g$ in $\overline{\C}\setminus\Delta$ such that
	\begin{equation}
		\label{szego_BVP}
		\begin{cases}
			g(u)\neq 0 \textrm{ for } u \in \overline{\C}\setminus\Delta;\\
			g(\infty)>0;\\
			\lim_{u \to x} |g(u)|^2 =  (\sqrt{(b-x)(x-a)}\,\mu'(x))^{-1},\quad \textrm{ a.e. on }\,\, \Delta.
		\end{cases}
	\end{equation}
	This problem has as solution $g(u) = \mathsf{G}(\mu, u)$ (for an expanded exposition see \cite[Ch.
\textsc{xvi}]{szego}).	

	When $h$ is a function on $\Delta$ such that $\ln h $ is integrable with respect to   $\D\eta_\Delta(x)$ we also write
	$$\mathsf{G}(h,u) = \exp\left[\frac{\sqrt{(u-b)(u-a)}}{2\pi}\int_\Delta\frac{\ln h(x) }{x-u}\D\eta_\Delta(x)\right],\quad u \in \overline{\mathbb{\C}} \setminus \Delta.$$
	These functions are related with outer functions (see \cite[Def. 17.14]{rudin}) whose analytic representation is
	$$g(h,u) = c\exp\left[\frac{\sqrt{(u-b)(u-a)}}{2\pi}\int_\Delta\frac{\ln h(x) }{x-u}\D\eta_\Delta(x)\right],\quad u \in \overline{\mathbb{\C}} \setminus \Delta,$$
	where $c\in\mathbb{T}$.

	Notice that in the definition of the Szeg\H{o} function (which is the square of an outer function) with respect to a measure $\mu$ we do not take the Radon-Nikodym derivative $\mu'$ of $\mu$ with respect to the Lebesgue measure $\D x$ but instead with respect to $\D\eta_\Delta(x)$ which is precisely $\sqrt{(b-x)(x-a)}\mu'(x)$.

\subsection{A starting point}

Let  $x_{2n,i}, 2n-i_n +1\leq i \leq 2n,$ denote the zeros of $w_{2n}$. If $i_n < 2n$ we
define $x_{2n,i} = \infty, 1\leq i \leq 2n - i_n$. Set
  	$$B_{2n}(u) := \prod_{i=1}^{2n}\frac{\Psi(u)-\Psi(x_{2n,i})}{1-\overline{\Psi(x_{2n,i})}\Psi(u)}.$$
When $x_{2n,i} = \infty$ the corresponding factor in the Blaschke product is replaced by $1/\Psi(u)$.

In  \cite[Theorem 4]{bernardo_lago2} a strong asymptotic result is given. We state it as a lemma for convenience of the reader and further reference in the paper.
\begin{lemma}
  	\label{LB_Blaschke}
  	Assume that $(\mu_n ,w_{2n} )_{n\geq 0}$ verifies $i)-iii)$  and $l_n$ is the $n$-th orthonormal polynomials associated with \eqref{def:Ln}. Then
  	\begin{equation}
  		\label{asintb}
  		\lim_n \frac{l_n^2(u)}{w_{2n}(u)}B_{2n}(u) =  \frac{1}{2\pi}\mathsf{G}^2(\mu,u),
  	\end{equation}
  	uniformly on compact subsets of $\Omega$.
\end{lemma}

We wish to point out that in \cite[Theorem 4]{bernardo_lago2} there is a typo when writing the condition $ii)$. There, it appears in terms of the Lebesgue measure $\D x$ instead of the Chebyshev measure $\D \eta_{\Delta}$. Except for that, the proof given is correct.

When $\mu_n = \mu$ is fixed and $w_{2n} \equiv 1$ (so that $B_{2n} \equiv 1/\Psi^{2n}$) we retrieve the standard result for the strong asymptotic of orthogonal polynomials with respect to $\mu \in \mathcal{S}(\Delta)$. The drawback of Lemma \ref{LB_Blaschke} is the appearance of the Blaschke product on the left hand side of \eqref{asintb}, but nothing can be done to simplify the expression unless some restriction is imposed on the asymptotic behavior of the sequence of polynomials $(w_{2n})_{n\geq 0}$.

If $\D \mu_n = h_n \D \overline{\mu}$, where $\overline{\mu}$ is a fixed measure satisfying Szeg\H{o}'s condition on $\Delta$, $(h_n)_{n\geq 0}$ is a sequence of positive continuous functions such that $\lim_n h_n = h$,  and $\lim_{n \to \infty} |w_{2n}(x)|\varphi^n(x) = 1/\psi(x) > 0$  uniformly on $\Delta$, the right hand side of \eqref{strong} becomes $\frac{1}{\sqrt{2\pi}}\mathsf{G}(\psi h  \overline{\mu},u)$.

\subsection{Proof of Theorem \ref{main_th}}
We begin with an auxiliary lemma.

\begin{lemma}
	\label{l1}
	Assume that the sequence of polynomials $(w_{2n})_{n \geq 0}$ verifies $iv)$. Then
	\begin{equation}
		\label{eq:a}
		\lim_{n \to \infty} C^{2n} \Phi^{2n}(u) \frac{B_{2n}(u)}{w_{2n}(u)} = \mathsf{G}^{-2}(\psi,u)
	\end{equation}
	uniformly on compact subsets of $\overline{\C} \setminus \Delta$, where $\Phi$ and $C$ are defined as in \eqref{PhiC}.
\end{lemma}

\begin{proof}
	Notice that
	\begin{equation}
		\label{eq:b}
		C^{2n} \Phi^{2n}(u) \frac{B_{2n}(u)}{w_{2n}(u)} = \left(\frac{C^{2n}
\Phi^{2n}(u)}{\Psi^{2n}(u)}\right)\left(\frac{\Psi^{2n}(u)B_{2n}(u)}{w_{2n}(u)}\right)
	\end{equation}
	and consider each factor in parenthesis on the right hand side separately.

	Define the function
 		\[
 		f_{2n,i}(u) := \frac{\Psi(u)}{u-x_{2n,i}}\frac{\Psi(u) - \Psi(x_{2n,i})}{1-\overline{\Psi(x_{2n,i})}\Psi(u)}.
 		\]
	 It is easy to verify that this function is holomorphic and never vanishes in $\overline{\C}\setminus\Delta$. Also, $|f_{2n,i}|$ can be extended continuously to $\Delta$ with boundary values $|f_{2n,i}(x)|=|x-x_{2n,i}|^{-1}, x \in \Delta$. Moreover,
	 	$$f_{2n,i}(u) = \frac{\Psi(u)}{u}\frac{1}{1-x_{2n,i}u^{-1}}\frac{1-\Psi(x_{2n,i})\Psi^{-1}(u)}{\Psi^{-1}(u) -
\overline{\Psi(x_{2n,i})}},$$
 	thus $f_{2n,i}(\infty) = -\Psi'(\infty)/\overline{\Psi(x_{2n,i})}$. As $|f_{2n,i}|$ is continuous and different from zero in $\overline{\C} $, it follows that $f_{2n,i}$ and $f_{2n,i}^{-1}$ are in $H_1(\overline{\C}\setminus\Delta)$ with respect to  the Chebyshev measure on $\Delta$; consequently, $f_{2n,i}$ is an outer function (see \cite[Chap. 17, Ex. 19]{rudin}).
 Then,
 	\begin{equation*}
 		f_{2n,i}(u) = c_i \exp\left[\frac{\sqrt{(u-a)(u-b)}}{\pi}\int_\Delta\frac{\ln|x-x_{2n,i}|}{x-u}\D\eta_\Delta(x)\right],
 	\end{equation*}
 	(see \eqref{Glimit}) where $c_i$ is a constant, $|c_i|= 1$. Should $w_{2n}$ be monic, an easy consequence of this representation is
	\begin{equation}
		\label{eq:c}
 		\frac{(\Psi^{2n}B_{2n})(u)}{w_{2n}(u)} = \prod_{i=1}^{2n}f_{2n,i}(u) =
 \exp\left[\frac{\sqrt{(u-a)(u-b)}}{\pi}\int_\Delta\frac{\ln |w_{2n}(x)|}{x-u}\D \eta_{\Delta}(x) \right].
 	\end{equation}
 	(The product of all the constants $c_i$ gives $1$.) If $w_{2n}$ is not monic then the same representation holds due to  the fact that for any positive constant $\kappa$
 		\[ \exp\left[\frac{\sqrt{(u-a)(u-b)}}{\pi}\int_\Delta\frac{\ln \kappa}{x-u}\D \eta_{\Delta}(x) \right] = \frac{1}{\kappa}.
 		\]

	On the other hand, $(C\Phi)^2/\Psi^2$ is analytic and different from zero in $\overline{\C}\setminus\Delta$. Moreover, $|C\Phi(x)/\Psi (x)|^2 = \exp(2\gamma - 2V_{\lambda_\varphi}(x))$, $x\in\Delta,$ and using the equilibrium condition $|C\Phi(x)/\Psi (x)|^2= \exp\left(-\ln\varphi(x)\right)= 1/\varphi(x)$, $x\in \Delta$. Consequently, $C^2\Phi^2/\Psi^2$ is an outer function and we have
 	\begin{equation}
 		\label{asymptotic3}
		\frac{C^{2n}\Phi^{2n}(u)}{\Psi^{2n}(u)} = \exp\left(\frac{\sqrt{(u-a)(u-b)}}{\pi}\int_\Delta\frac{n\ln\varphi(x)}{x-u}\D \eta_{\Delta}(x)  \right).
 	\end{equation}

	Putting together \eqref{eq:b}, \eqref{eq:c}, and \eqref{asymptotic3}, we have
		\[ C^{2n} \Phi^{2n}(u) \frac{B_{2n}(u)}{w_{2n}(u)} = \exp\left(\frac{\sqrt{(u-a)(u-b)}}{\pi}	\int_\Delta\frac{\ln
(|w_{2n}(x)|\varphi^n(x))}{x-u}\D \eta_{\Delta}(x)  \right).
		\]
	To deduce \eqref{eq:a} it remains to use \eqref{rest_var_w} and the definition of $\mathsf{G}(\psi,u)$.
\end{proof}

Now, Theorem \ref{main_th} is easy to derive. Indeed
	\[ \frac{l_n^2(u)}{C^{2n}\Phi^{2n}(u)} = \frac{l_n^2(u)B_{2n}(u)}{w_{2n}(u)} \frac{w_{2n}(u)
}{C^{2n}\Phi^{2n}(u)B_{2n}(u)}.
	\]
As $n\to \infty$, the limit of the first factor on the right is given by Lemma \ref{LB_Blaschke} and that of the second one by Lemma \ref{l1}. The proof of \eqref{strong}  has been concluded.

Next, we deduce the asymptotic behavior of the monic orthogonal polynomials $L_n$ and the leading coefficients $\tau_n$ of $l_n$. It is easy to see that
		\[ \Phi(u) = e^{-v_{\lambda_\varphi}(u)} = u + \mathcal{O}(1),\qquad u \to \infty.
		\]
Using \eqref{strong} at $u=\infty$, we obtain \eqref{constante}. Then, \eqref{asintc} follows directly from \eqref{strong} and \eqref{constante}. \hfill $\Box$

\subsection{Proof of Theorem \ref{cor:a}}

We wish to express the orthogonality relations of the polynomials $p_n$ in such a way that we can apply Theorem
\ref{main_th}.

Notice that $|\Phi_\tau(x)|^{-1} = \exp V_{\tau}(x)$. According to \cite[Theorem 10.2]{Totik} (see also \cite[Lemma
9.1]{Totik}), there exists a sequence of  polynomials $(H_{n-1})_{n\geq 0}, \deg H_{n-1} \leq n-1,$ which do not vanish on
$\Delta$ whose zeros verify condition $iii)$   (see assertion on page 94 in \cite{Totik}) such that
\begin{equation}
	\label{eq:d}
	|H_{n-1}(x)/\Phi_\tau^n(x)| \leq 1, \qquad x \in (a,b),
\end{equation}
\begin{equation}
	\label{eq:e}
	\lim_n |H_{n-1}(x)/\Phi_\tau^n(x)| =1,
\end{equation}
uniformly on compact subsets of $(a,b)$, and
\begin{equation}
	\label{eq:f}
	\lim_n \int_a^b \ln (|H_{n-1}(x)/\Phi_\tau^n(x)|)\D\eta_\Delta(x) = 0.
\end{equation}

Now, the orthogonality relations satisfied by the polynomials $p_n$ can be rewritten as
	\[ \int p_m(x) p_n(x) \frac{|H_{n-1}^2(x)|}{|\Phi_{\tau}^{2n}(x)|}\frac{\D \mu_n(x)}{|H_{n-1}^2(x)|} = \left\{
 	\begin{array}{ll}
 	0, & m < n, \\
 	1, & m = n.
 	\end{array}\right.
 	\]
Let us check that the sequence
	\[ \left( \frac{|H_{n-1}^2(x)|\D \mu_n}{|\Phi_{\tau}^{2n}(x)|}, H_{n-1}^2(x) \right)_{n \geq 0}
 	\]
verifies $i)-iv)$. Indeed, the zeros of the polynomials $H_{n-1}$, and thus of the polynomials $H_{n-1}^2,$ $\deg H_{n-1}^2 \leq 2n$, verify condition $iii)$. On the other hand, \eqref{eq:d}, \eqref{eq:e}, and condition $i)$ for the sequence of measures $(\mu_n)_{n\geq 0}$ imply condition $i)$ for the sequence of measures $(|H_{n-1}^2|\D \mu_n/|\Phi_{\tau}^{2n}|)_{n\geq 0}$ and
 	\[ \liminf_n \int \ln \left(\frac{|H_{n-1}^2(x)|}{|\Phi_{\tau}^{2n}(x)|}\mu_n'(x) \right)\D\eta_\Delta(x) \geq \int \ln\mu'(x) \D\eta_\Delta(x);\]
 therefore, $ii)$ takes place.
Take $w_{2n} = H_{n-1}^2$ and $\varphi = e^{2V_\tau}$. Using \eqref{eq:f} we obtain $iv)$ with $\psi \equiv 1$. The
equilibrium condition corresponding to this case is the trivial one
 	\[ V_\tau(x) - V_\tau(x) \equiv 0
 	\]
and the equilibrium constant is $\gamma = 0$; therefore $C= 1$. Applying Theorem \ref{main_th} the thesis of Theorem
\ref{cor:a} readily follows.  \hfill $\Box$

\subsection{Applications to rational approximation}
Let $\mu$ be a positive measure with $\supp\mu=\Delta$ that satisfies Szeg\H{o}'s condition. (In this section we take $h_n \equiv 1, n \geq 0$.) The  corresponding Markov function, also known as the Cauchy transform of the measure $\mu$, is defined as
\begin{equation}
	\label{cauchy}
	\hat{\mu}(z) := \int_\Delta\frac{\D \mu (x)}{z-x}.
\end{equation}

Consider a sequence of polynomials $(w_{2n})_{n\geq 0}$ as above, positive on $\Delta$. It is well known that for each $n \geq 1$, there exists a rational function $R_n =\frac{L^*_{n-1} }{L_n }$, $\deg L^*_{n-1}\leq n-1$ and $\deg L_n\leq n$ such that
\begin{equation*}
	\frac{(L_n  \hat{\mu}-L^*_{n-1})(z)}{w_{2n}(z)} = \frac{A_n}{z^{n+1}} + \cdots, \qquad z\rightarrow\infty
\end{equation*}
where the function on the left hand side is analytic in $\overline{\C}\setminus\Delta$. The fraction $R_n$ is called the $n$-th multi-point Padé approximant of $\hat{\mu}$ with respect to $w_{2n}$. It is well known and easy to prove that $L_n$ is an $n$-th orthogonal polynomial with respect to the varying measure $\mu/w_{2n}$ and it can be taken to be monic. The remainder of $\hat{\mu}-R_n$ has the integral expression (see \cite{gonchar_lago})
\begin{equation*}
	(\hat{\mu}-R_n)(z) = \frac{w_{2n}(z)}{L_n^2(z)}\int_\Delta \frac{L_n^2(x) \D \mu (x)}{w_{2n}(x)(z-x)} =
\frac{w_{2n}(z)}{l_n^2(z)}\int_\Delta\frac{l^2_n(x)\D \mu(x)}{w_{2n}(x)(z-x)},
\end{equation*}
where $l_n$ denotes the corresponding orthonormal polynomial.

Taking into account \cite[Theorem 8]{bernardo_lago1}, we know that
\begin{equation*}
	\lim_{n} \int_\Delta\frac{l^2_n(x) \D \mu(x)}{|w_{2n}(x)|(z-x)} = \frac{1}{\sqrt{(z-b)(z-a)}},
\end{equation*}
uniformly on compact subsets of $\overline{\C} \setminus \Delta$, where the square root is chosen to be positive when $z > b$. So, a direct consequence of Lemma \ref{LB_Blaschke} and Theorem \ref{main_th} is the next result.
\begin{corollary}
Assume that $i)-iii)$ take place where $h_n \equiv 1, n\geq 0$. We have
	\begin{equation*}
		\lim_{n }\frac{(\hat{\mu}-R_n)(z)}{B_{2n}(z)} = \frac{2\pi\mathsf{G}^{-2}(\mu,z)}{\sqrt{(z-a)(z-b)}}.
	\end{equation*}
	If, additionally, $iv)$ holds and $\supp \lambda_{\varphi} = \Delta$, then
	\begin{equation*}
		\lim_{n}	\frac{(C\Phi)^{2n}(z)(\hat{\mu}-R_n)(z)}{w_{2n}(z)} =
\frac{2\pi\mathsf{G}^{-2}(\psi\mu,z)}{\sqrt{(z-a)(z-b)}}.
	\end{equation*}
	The limits are uniform on compact subsets of $\Omega$.
\end{corollary}

\section{Biorthogonal polynomials and multi level Hermite Pad\'e polynomials}

\subsection{Multilevel HP polynomials}
Let $\Delta_1, \Delta_2$ be non-intersecting closed intervals of the real line. Let $(\sigma_1, \sigma_2) \in \mathcal{M}(\mathbf{\Delta})$ where $\mathbf{\Delta} = (\Delta_1,\Delta_2)$. Using the differential notation, we define a third measure $s_{1,2}$  by
\[ \D s_{1,2}(x) := \widehat{\sigma}_2(x) \D \sigma_1(x),
\]
where $\widehat{\sigma}_2$ is the Markov function of $\sigma_2$ (see \eqref{cauchy}). Inverting the role of the measures we define similarly $s_{2,1}, \D s_{2,1}(x) = \widehat{\sigma}_1(x) \D \sigma_2(x)$.
The pair $\mathcal{N}(\sigma_1,\sigma_2):=(s_{1,1},s_{1,2})$, where $s_{1,1} = \sigma_1$, is called the Nikishin system generated by $(\sigma_1,\sigma_2)$.

Notice that the order in which the measures are taken is important and $\mathcal{N}(\sigma_1,\sigma_2)\not=\mathcal{N}(\sigma_2,\sigma_1)$. General Nikishin systems of $m\geq 2$ measures were first introduced in \cite{nikishin}. They verify interesting properties (see, for example, \cite{Ulises_Lago}, \cite{ulises_lago_2}) and have found numerous applications in different areas of mathematics. In particular, Nikishin systems of two measures  appear in the analysis of the two matrix model \cite{Ber,BGS} and in finding discrete solutions of the Degasperis-Procesi equation \cite{Bertola:CBOPs} through a Hermite-Pad\'e approximation problem for two discrete measures. Motivated in \cite{Bertola:CBOPs}, the approximation problem was extended  in \cite{LMS} for arbitrary $m \geq 2$ and general measures proving the convergence of the method. We will focus on the case of two measures.

For each $n \in \mathbb{N},$ there exists a vector polynomial $(a_{n,0},a_{n,1}, a_{n,2}),$ not identically equal to zero, with $\deg a_{n,0}\leq n-1$, $\deg a_{n,1}\leq n-1,$ and $\deg a_{n,2}\leq n$,  that satisfies
\begin{align}
\mathcal{A}_{n,0}(z):=\left(a_{n,0}-a_{n,1}\widehat{s}_{1,1}+a_{n,2}\widehat{s}_{1,2}\right)(z)=\mathcal{O}(1/z^{n+1}) \label{JLS1},\\
\mathcal{A}_{n,1}(z):= \left(-a_{n,1}+a_{n,2}\widehat{s}_{2,2}\right)(z)=\mathcal{O}(1/z). \label{JLS2}
\end{align}
Here and below, the symbol $\mathcal{O}(*)$ is taken as $z \to \infty$. By extension we take $\mathcal{A}_{n,2} \equiv a_{n,2}$. The polynomials $a_{n,0}, a_{n,1}, a_{n,2}$ are called multilevel Hermite-Pad\'e polynomials.

It can be shown that $\deg a_{n,2}= n$ and the vector polynomial can be normalized taking $a_{n,2}$ monic. With this normalization $(a_{n,0},a_{n,1}, a_{n,2})$ is unique. Moreover, all the zeros of $a_{n,2}$ are simple and lie in the interior $\stackrel{\circ}{\Delta}_2$ (with the Euclidean topology of $\R$) of the interval $\Delta_2$. For more details, see
\cite[Theorem 1.4]{LMS} and Lemma \ref{l2} below.

Combining  Cauchy's theorem, Fubini's theorem, and Cauchy's integral formula, from \eqref{JLS1} it follows that
\[ \int x^\nu \mathcal{A}_{n,1}(x) \D \sigma_1(x) = 0, \qquad \nu=0,\ldots,n-1,
\]
and from \eqref{JLS2} we get the integral representation
\[ \mathcal{A}_{n,1}(x) = \int \frac{a_{n,2}(y)\D \sigma_2(y)}{x -y}.
\]
Therefore,
\[ \int \int \frac{x^{\nu} a_{n,2}(y)}{x-y} \D \sigma_1(x) \D \sigma_2(y) = 0, \qquad \nu =0,\ldots,n-1.
\]
Consequently, $a_{n,2}$ normalized to be monic, verifies the same orthogonality relations as the biorthogonal polynomial
$Q_n$ (see \eqref{biort}) and coincides with it.

Analogously, for each $n \in \mathbb{N},$ there exists a vector polynomial $(b_{n,0},b_{n,1}, b_{n,2}),$ not identically equal to zero, with $\deg b_{n,0}\leq n-1$, $\deg b_{n,1}\leq n-1,$ and $\deg b_{n,2}\leq n$,  that satisfies
\begin{align}
\mathcal{B}_{n,0}(z):=
\left(b_{n,0}-b_{n,1}\widehat{s}_{2,2}+b_{n,2}\widehat{s}_{2,1}\right)(z)=\mathcal{O}(1/z^{n+1}) \label{JLS1*},\\
\mathcal{B}_{n,1}(z):= \left(-b_{n,1}+b_{n,2}\widehat{s}_{1,1}\right)(z)=\mathcal{O}(1/z). \label{JLS2*}
\end{align}
By extension we take $\mathcal{B}_{n,2} \equiv b_{n,2}$. Normalizing $b_{n,2}$ to be monic, we have $b_{n,2} = P_n$ (the other biorthogonal polynomial in \eqref{biort}).

Therefore, in order to prove Theorem \ref{mainbiort}, we need to find the strong asymptotic of the sequences of polynomials $(a_{n,2})_{n\geq 0}$ and $(b_{n,2})_{n\geq 0}$. Because of the symmetry of the problem, it suffices to analyze the first sequence and the results for the second one are immediate.

Indeed, we will give the strong asymptotic of the forms $\mathcal{A}_{n,0}, \mathcal{A}_{n,1}$ and the polynomials $a_{n,0}, a_{n,1}, a_{n,2}$, as $n\to \infty$, under the assumption that the generating measures $\sigma_1, \sigma_2$ are in the Szeg\H{o} class; that is, $(\sigma_1,\sigma_2) \in \mathcal{S}(\mathbf{\Delta})$ (see \eqref{szego}). For general Nikishin systems of $m\geq 2$ measures, the logarithmic and ratio asymptotic of ML HermitePad\'e polynomials was studied in \cite{ULS} (see also \cite{Lysov}).

\subsection{Some useful properties}

The forms $\mathcal{A}_{n,k}, k=0,1,2,$  are interlinked and satisfy interesting orthogonality relations which will be of great use. The following result, is a special case $(m=2)$ of \cite[Lemma 2.4]{ULS}. It is stated here for convenience of the reader.

\begin{lemma}
	\label{l2}
	Consider the Nikishin system $\mathcal{N}(\sigma_1, \sigma_2)$. For each fixed $n \in \mathbb{Z}_+$ and $j=1,2$, $\mathcal{A}_{n,j}$ has exactly $n$ zeros in $\mathbb{C} \setminus \Delta_{j+1}$ they are all simple and lie in $\stackrel{\circ}{\Delta}_{j}$ $({\Delta}_{3} = \varnothing)$. $\mathcal{A}_{n,0}$ has no zero in $\mathbb{C} \setminus \Delta_{1}$. Let $Q_{n,j}, j=1,2,$ denote the monic polynomial of degree $n$ whose zeros are those of  $\mathcal{A}_{n,j}$ in $\Delta_j$. For $j=0,1 $,
	\begin{equation}
		\label{int1}
		\frac{\mathcal{A}_{n,j}(z)}{Q_{n,j}(z)} = \int   \frac{\mathcal{A}_{n,j+1}(x)}{z-x}  \frac{ {\rm
		d}\sigma_{j+1}(x)}{Q_{n,j}(x)},
\end{equation}
where $Q_{n,0} \equiv 1$, and
\begin{equation}
	\label{int2}
	\int x^{\nu}  \mathcal{A}_{n,j+1}(x)  \frac{{\rm d}\sigma_{j+1}(x)}{Q_{n,j}(x)} = 0, \qquad \nu = 0,\ldots, n-1.
\end{equation}
\end{lemma}

The orthogonality relations involving the linear forms $\mathcal{A}_{n,j}$ stated in \eqref{int2} can be rewritten in terms of orthogonal polynomials with varying measures. That is
\begin{equation} \label{int3}
0 = \int x^{\nu}  Q_{n,2}(x)  \frac{{\rm d}\sigma_{2}(x)}{Q_{n,1}(x)} , \qquad \nu = 0,\ldots, n-1.
\end{equation}
and
\begin{equation} \label{int4}
0 = \int x^{\nu}  Q_{n,1}(x) \mathcal{H}_{n,1}(x)   \frac{{\rm d}\sigma_{1}(x)}{Q_{n,2}(x)}, \qquad \nu = 0,\ldots, n-1.
\end{equation}
where, using \eqref{int1} with $j=0$ and \eqref{int3}
\begin{equation} \label{hn1} \mathcal{H}_{n,1}(z) := \frac{Q_{n,2}(z) \mathcal{A}_{n,1}(z)}{Q_{n,1}(z)} = Q_{n,2}(z) \int
\frac{Q_{n,2}(x)}{z-x}  \frac{ {\rm d}\sigma_{2}(x)}{Q_{n,1}(x)} = \int   \frac{Q_{n,2}^2(x)}{z-x}  \frac{ {\rm
d}\sigma_{2}(x)}{Q_{n,1}(x)}.
\end{equation}

\begin{proposition} \label{prop1} There is a unique pair of monic polynomials with real coefficients $(Q_{n,1},Q_{n,2})$
each one of degree $n$, whose zeros lie in $\C \setminus \Delta_2$ and $\C \setminus \Delta_1$, respectively, satisfying
\eqref{int3}-\eqref{int4} with
\[\mathcal{H}_{n,1}(z)  = \int   \frac{Q_{n,2}^2(x)}{z-x}  \frac{ {\rm d}\sigma_{2}(x)}{Q_{n,1}(x)}.
\]
\begin{proof} The existence of such polynomials is guaranteed by Lemma \ref{l2}. We must show that if $(Q_{n,1},
Q_{n,2})$ is a pair of monic polynomials of degree $n$ which satisfy \eqref{int3}-\eqref{int4} with $\mathcal{H}_{n,1}$ as
indicated then we can construct forms $\mathcal{A}_{n,0},\mathcal{A}_{n,1},\mathcal{A}_{n,2}$ verifying
\eqref{JLS1}-\eqref{JLS2} whose zeros are those of the polynomials $Q_{n,1},Q_{n,2}$.

So, let $(Q_{n,1},Q_{n,2})$ be an arbitrary pair of monic polynomials of degree $n$ which satisfy \eqref{int3}-\eqref{int4}.
Take $\mathcal{A}_{n,2}= a_{n,2} := Q_{n,2}$ and
\[ a_{n,1}(z) :=  \int \frac{Q_{n,2}(z) - Q_{n,2}(x)}{z-x}  \D \sigma_2(x).
\]
Obviously, $a_{n,1}$ is a polynomial of degree $\leq n-1$. Rearranging this equality
and using \eqref{int3}, we get
\[ \mathcal{A}_{n,1}(z) := (-a_{n,1} + a_{n,2}  \widehat{s}_{2,2})(z) = \int \frac{(Q_{n.1} Q_{n,2})(x)}{z-x} \frac{\D
\sigma_2(x)}{Q_{n,1}(x)} = Q_{n,1}(z) \int \frac{ Q_{n,2}(x)}{z-x} \frac{\D \sigma_2(x)}{Q_{n,1}(x)}.
\]
The first equality tells us that $\mathcal{A}_{n,1}(z) = \mathcal{O}(1/z)$, so that \eqref{JLS2} takes place, and the last
equality implies that the zeros of $\mathcal{A}_{n,1}$ in $\C \setminus \Delta_2$ coincide with the simple roots that
$Q_{n,1}$ has in the interior of $\Delta_1$. Moreover, these relations together with \eqref{int3}-\eqref{int4} imply that for
each $\nu=0,1,\ldots,n-1$
\[ \int x^{\nu} \mathcal{A}_{n,1}(x) \D \sigma_1(x) = \int x^{\nu} Q_{n,1}(x) \int \frac{ Q_{n,2}(t)}{x-t} \frac{\D
\sigma_2(t)}{Q_{n,1}(t)} \D \sigma_1(x) =
\]
\[ \int x^{\nu} Q_{n,1}(x) \int \frac{ Q_{n,2}^2(t)}{x-t} \frac{\D \sigma_2(t)}{Q_{n,1}(t)} \frac{\D \sigma_1(x)}{Q_{n,2}(x)} =
\int x^{\nu} Q_{n,1}(x)   \mathcal{H}_{n,1}(x) \frac{\D \sigma_1(x)}{Q_{n,2}(x)} =0.
\]
These orthogonality relations verified by $\mathcal{A}_{n,1}$ in turn imply that
\begin{equation} \label{orderAn0}
\int \frac{\mathcal{A}_{n,1}(x)}{z-x} \D \sigma_1(x) = \frac{1}{z^n} \int \frac{x^n \mathcal{A}_{n,1}(x)}{z-x}\D \sigma_1(x) =
\mathcal{O}(1/z^{n+1})
\end{equation}

Using the definition of $\mathcal{A}_{n,1}(x)$, we get
\[ a_{n,0}(z) := a_{n,1}(z) \widehat{s}_{1,1}(z) - a_{n,2}(z) \widehat{s}_{1,2}(z) +
\int \frac{  \mathcal{A}_{n.1}(x)}{z-x} \D \sigma_1(x) =
\]
\[ \int \frac{a_{n,1}(z) - a_{n,1}(x)}{z-x} \D \sigma_1(x) - \int \frac{a_{n,2}(z) - a_{n,2}(x)}{z-x} \D s_{1,2}(x),
\]
which is obviously a polynomial of degree $\leq n-1$. Rearranging this equality and taking account of \eqref{orderAn0}, it
follows that
\[ \mathcal{A}_{n,0}(z) :=  a_{n,0}(z) - a_{n,1}(z) \widehat{s}_{1,1}(z) + a_{n,2}(z) \widehat{s}_{1,2}(z) =
\int \frac{  \mathcal{A}_{n.1}(x)}{z-x} \D \sigma_1(x) = \mathcal{O}(1/z^{n+1}).
\]
Thus, $\mathcal{A}_{n,0}$ verifies \eqref{JLS1}.

From our findings, we deduce that the vector polynomial $(a_{n,0}, a_{n,1}, a_{n,2})$ defined previously is the unique
solution of \eqref{JLS1}-\eqref{JLS2}. In particular,
$a_{n,2} = Q_{n,2}$ is uniquely determined and by \eqref{int4} so is $Q_{n,1}$ since the measure $(\mathcal{H}_{n,1}
\D \sigma_1)/ Q_{n,2}$ has constant sign on $\Delta_1$. We are done.
\end{proof}

\end{proposition}

\subsection{Normalization}
	Set
	\begin{equation}
		\label{kappa} \kappa_{n,2}^{-2} := \int    Q_{n,2}^2(x)  \frac{{\rm d}\sigma_{2}(x)}{|Q_{n,1}(x)|}, \qquad
(\kappa_{n,1}\kappa_{n,2})^{-2} :=  \int  Q_{n,1}^2(x)    \frac{|\mathcal{H}_{n,1}(x)|{\rm d}\sigma_{1}(x)}{|Q_{n,2}(x)|}.
	\end{equation}
	Take
	\begin{equation}
		\label{Qs}
		q_{n,1} := \kappa_{n,1}Q_{n,1}, \qquad q_{n,2} := \kappa_{n,2}Q_{n,2}, \qquad h_{n,1} :=\kappa_{n,2}^2
\mathcal{H}_{n,1}.
	\end{equation}
	Notice that
	\[
	\kappa_{n,1}^{-2} :=  \int  Q_{n,1}^2(x)    \frac{|h_{n,1}(x)|{\rm d}\sigma_{1}(x)}{|Q_{n,2}(x)|}.
	\]
	We can rewrite \eqref{int3}-\eqref{int4} as
	\begin{equation}
		\label{int5}
		0 = \int x^{\nu}  q_{n,2}(x)  \frac{{\rm d}\sigma_{2}(x)}{|Q_{n,1}(x)|} = 0, \qquad \nu = 0,\ldots, n-1,
	\end{equation}
	and
	\begin{equation}
		\label{int6}
		0 = \int x^{\nu}  q_{n,1}(x)    \frac{|{h}_{n,1}(x)|{\rm d}\sigma_{1}(x)}{|Q_{n,2}(x|)}, \qquad \nu = 0,\ldots, n-1.
	\end{equation}
	We also have
	\begin{equation}
		\label{int7}
 		\int   q_{n,2}^2(x)  \frac{{\rm d}\sigma_{2}(x)}{|Q_{n,1}(x)|} = 1,
	\end{equation}
	and
	\begin{equation}
		\label{int8}
 		\int  q_{n,1}^2(x)    \frac{|{h}_{n,1}(x)|{\rm d}\sigma_{1}(x)}{|Q_{n,2}(x)|} = 1.
	\end{equation}
	Consequently, $q_{n,1}$ and $q_{n,2}$ are the $n$-th orthonormal polynomials with respect to the varying measures
$\frac{|{h}_{n,1} |{\rm d}\sigma_{1} }{|Q_{n,2} |}$ and $\frac{{\rm d}\sigma_{2} }{|Q_{n,1} |}$, respectively. Recall that the
zeros of $Q_{n,j}$ lie in $\stackrel{\circ}{\Delta}_j = (a_j,b_j), j=1,2$.

From \cite[Theorem 8]{bernardo_lago1} it follows that if $\sigma_2' > 0$ a.e. on $\Delta_2, $ then  for any bounded measurable function $g_2$ on $\Delta_2,$
 \begin{equation} \label{int9}
\lim_n \int    \frac{g_2(x) q_{n,2}^2(x){\rm d}\sigma_{2}(x)}{|Q_{n,1}(x)|} =\frac{1}{\pi}\int_{a_2}^{b_2}  g_2(x)\D\eta_{\Delta_2}(x).
\end{equation}
Taking into account \eqref{hn1} and using \eqref{int9} with $g_2(x) = |t-x|^{-1}, t \in \Delta_1,$ and \eqref{hn1}, it follows that
 \begin{align} \label{int12}
\lim_n |h_{n,1}(t)| =& \lim_n \int    \frac{q_{n,2}^2(x){\rm d}\sigma_{2}(x)}{|t-x||Q_{n,1}(x)|} =\\
                    =&\frac{1}{\pi}\int_{a_2}^{b_2}  \frac{\D\eta_{\Delta_2}(x)}{|t-x|}= \frac{1}{\sqrt{|t-a_2||t-b_2|}} =: h(t),\nonumber
\end{align}
uniformly for $t \in \Delta_1$. Then \eqref{int8}, \eqref{int12}, and \cite[Theorem 8]{bernardo_lago1} imply that if $\sigma_1' > 0$ a.e. on $\Delta_1$, then for any bounded Borel measurable function $g_1$ on $\Delta_1$ we have
\begin{equation} \label{int10}
\lim_n \int    \frac{g_1(x) q_{n,1}^2(x)|{h}_{n,1}(x)|{\rm d}\sigma_{1}(x)}{|Q_{n,2}(x)|} =\frac{1}{\pi}\int_{a_1}^{b_1}
g_1(x) \D\eta_{\Delta_1}(x).
\end{equation}

\subsection{The comparison functions}

The logarithmic asymptotic of general ML Hermite-Pad\'e polynomials was studied in \cite[Section 3]{ULS}. In particular, it
was proved that this asymptotic behavior can be described in terms of the solution of a
 vector equilibrium problem which, in the case we are dealing with, reduces to finding a pair of probability measures
 $(\lambda_1,\lambda_2), \supp \lambda_1 \subset \Delta_1, \supp \lambda_2 \subset \Delta_2,$ and a pair of constants
 $(\gamma_1,\gamma_2)$ such that
\begin{equation}
\label{vector_equil}
\begin{cases}
V_{\lambda_1}(x)-\frac{1}{2}V_{\lambda_2}(x)\equiv \gamma_1, &\qquad x\in\Delta_1,\\
V_{\lambda_2}(x)-\frac{1}{2}V_{\lambda_1}(x)\equiv \gamma_2, &\qquad x\in\Delta_2.
\end{cases}
\end{equation}
It is well known that this problem has a unique solution.  From \cite[Theorem 3.4]{ULS} it follows that if  $\sigma_1$ and
$\sigma_2$ are regular measures (for the definition and properties of regular measures, see \cite[Chap. 3]{ST})  then for
$k=1,2$ we have
\begin{equation} \label{weak}
\lim_n |Q_{n,k}|^{1/n} = \exp(- V_{\lambda_k}),\qquad \lim_n \kappa_{n,k}^{1/n} = \gamma_k,
\end{equation}
where the first limit is uniform on compact subsets of $\C \setminus \Delta_k$.

Since strong asymptotic implies weak asymptotic, \eqref{weak} reveals that the functions with which one must compare
the polynomials $q_{n,1}, q_{n,2}$ in order to have strong asymptotic (should it exist) are tightly connected with the
potentials $V_{\lambda_1}, V_{\lambda_2}$ and the constants $\gamma_1, \gamma_2$. With this in mind (see
\eqref{PhiC}), we define
\begin{equation}\label{compfunc}
\Phi_k(z) := e^{-(V_{\lambda_k} + i \widetilde{V}_{\lambda_k})(z)}, \qquad C_k := e^{\gamma_k}, \qquad k=1,2,
\end{equation}
where $\widetilde{V}_{\lambda_k}$ denotes the harmonic conjugate of $ {V}_{\lambda_k}$ in $\C \setminus \Delta_k$.
For a different expression of the comparison functions see \eqref{otro}.

In \eqref{int3}-\eqref{int4} we see that the orthogonality relations verified by the polynomials $Q_{n,1}, Q_{n,2}$ are interconnected. This prevents the direct use of Theorem \ref{main_th} to obtain their asymptotic because to give the asymptotic of one of the sequences one must know that of the second, and viceversa. So, as indicated in the introduction, we will follow an indirect approach devised by A.I. Aptekarev to attack  analogous problems in \cite{sasha2} and \cite{sasha}.

\subsection{Prescribed asymptotic behavior}
	An important ingredient of the method consists in being capable of producing a sequence of functions of the form $P_{n,k}/\Phi_k^n, k=1,2$, where $P_{n,k}$ is a polynomial of degree $n$, whose limit is a predetermined Szeg\H{o} function.

	Let $(\lambda_1,\lambda_2)$ be the solution of the vector equilibrium problem \eqref{vector_equil}. From \cite[Theorem 1.34]{DKL}  it follows that $\D \lambda_k = v_k \D x$ on $\Delta_k, k=1,2$, and the weights $v_1, v_2$ verify the assumptions relative to $v$ in Theorem \ref{cor:a} on the intervals $\Delta_1, \Delta_2$, respectively. In the sequel
 	\[ \Omega_k := \overline{\C} \setminus \Delta_k, \qquad k=1,2.
 	\]

 \begin{proposition} \label{prop2}
 Assume that $(\mu_1,\mu_2) \in \mathcal{S}(\mathbf{\Delta})$
 and for each $n \geq0$, $(\tilde{q}_{n,1},\tilde{q}_{n,2})$  is the pair of polynomials of degree $n$ such that
 \begin{equation}\label{poltilde}  \int \tilde{q}_{n,k}(x) \tilde{q}_{m,k}(x) \frac{\D \mu_k(x)}{|\Phi_{k}^{2n}(x)|} =
 \left\{
 \begin{array}{cc}
 0, & 0\leq m <n, \\
 1, & m=n,
 \end{array}
 \right.  \qquad k=1,2.
 \end{equation}
 Then
 \begin{equation}\label{predet}
 \lim_{n\to \infty} \frac{\tilde{q}_{n,k}(z)}{\Phi_k^{n}(z)} = \frac{\mathsf{G}(\mu_k,z)}{\sqrt{2\pi}},
 \end{equation}
 and
 \begin{equation}\label{predet*}
 \lim_{n\to \infty} \frac{\tilde{Q}_{n,k}(z)}{\Phi_k^{n}(z)} = \frac{\mathsf{G}(\mu_k,z)}{\mathsf{G}(\mu_k,\infty) },
 \end{equation}
 uniformly on each compact subset of $\Omega_k, k=1,2$, where the $\Phi_k$ were introduced in \eqref{compfunc} and
 $\tilde{Q}_{n,k}, k=1,2$ is $\tilde{q}_{n,k}$ renormalized to be monic.
 \end{proposition}

 \begin{proof} As was mentioned above, \cite[Theorem 1.34]{DKL} guarantees that the components of the equilibrium
 measures $(\lambda_1,\lambda_2)$ are absolutely continuous with respect to the Lebesgue measure on the
 corresponding intervals and their weights $v_1,v_2$ verify \eqref{cond:a} with parameters $\beta=\beta_0= -1/2$ on the
 intervals $\Delta_1,\Delta_2$, respectively. The   assumptions of Theorem \ref{cor:a} are   verified and
 \eqref{predet} follows directly from \eqref{strong:a}. If $\tilde{\kappa}_{n,k}$ is the leading coefficient
 of $\tilde{q}_{n,k}$, applying \eqref{predet} at $z = \infty$ we get
 \[ \lim_n \tilde{\kappa}_{n,k} = \frac{\mathsf{G}(\mu_k,\infty)}{\sqrt{2\pi}}
 \]
 and \eqref{predet*}  follows at once.
 \end{proof}

\subsection{The operator $\tilde{T}_n$}

\begin{definition}
\label{oper_Tn}	
Let $(\sigma_1,\sigma_2) \in \mathcal{M}(\mathbf{\Delta})$. Let $(\tilde{h}_n)_{n\geq 0}$ be a sequence of positive
continuous functions on $\Delta_1$ such that
\begin{equation}
	\label{hn}
	\lim_n \tilde{h}_n = \tilde{h} >0,
\end{equation}
uniformly on $\Delta_1$. Let $\mathcal{P}_{n,k}, k=1,2,$ be the set of all monic polynomials with real coefficients of
degree $n$ whose zeros  lie in $\C \setminus \Delta_2$ when $k= 1$  and in $\C \setminus \Delta_1$ when $k=2$.
Define an operator
\[\tilde{T}_n : \mathcal{P}_{n,1}\times \mathcal{P}_{n,2}\longrightarrow \mathcal{P}_{n,1}\times \mathcal{P}_{n,2}\]
where, for every $(\hat{Q}_{n,1},\hat{Q}_{n,2}) \in \mathcal{P}_{n,1}\times \mathcal{P}_{n,2}$
\begin{equation} \label{Tn} \tilde{T}_n(\hat{Q}_{n,1},\hat{Q}_{n,2}) :=   (  {Q}^*_{n,1}, {Q}^*_{n,2}) ,
\end{equation}
being $(  {Q}^*_{n,1}, {Q}^*_{n,2})$ the unique pair of monic polynomials of degree $n$ which satisfies
\begin{equation*}
0 = \int x^{\nu}  Q^*_{n,2}(x)  \frac{{\rm d}\sigma_{2}(x)}{|\hat{Q}_{n,1}(x)|} , \qquad \nu = 0,\ldots, n-1,
\end{equation*}
and
\begin{equation*}
0 = \int x^{\nu}  Q^*_{n,1}(x) \tilde{h}_{n}(x)   \frac{{\rm d}\sigma_{1}(x)}{|\hat{Q}_{n,2}(x)|}, \qquad \nu = 0,\ldots, n-1.
\end{equation*}
Set
\begin{equation} \label{kappa*} (\kappa^*_{n,2})^{-2} := \int (Q^*_{n,2}(x))^2   \frac{{\rm
d}\sigma_{2}(x)}{|\hat{Q}_{n,1}(x)|}, \qquad (\kappa^*_{n,1})^{-2} :=  \int  (Q^*_{n,1}(x))^2      \frac{\tilde{h}_{n}(x){\rm
d}\sigma_{1}(x)}{|\hat{Q}_{n,2}(x)|}.
\end{equation}
\end{definition}
Let us prove the continuity of the operator $\tilde{T}_n$, for each fixed $n \geq 0$. Given $\mathbf{Q}_1 = (Q_{1,1},Q_{1,2})$, $\mathbf{Q}_2 = (Q_{2,1},Q_{2,2}) \in \mathcal{P}_{n,1}\times \mathcal{P}_{n,2} $ define the metric $d_n$ as follows
\[ d_n(\mathbf{Q}_1,\mathbf{Q}_2) = \max \{ \|Q_{1,1} - Q_{1,2}\|_{\Delta_2}, \|Q_{2,1} - Q_{2,2}\|_{\Delta_1} \}
\]
where $\|\cdot\|_\Delta$ denotes the sup-norm on the interval $\Delta$. Suppose that
\[ \lim_{m\to \infty}d_n(\mathbf{Q}_{m},\mathbf{Q}) = 0
\]
and $ \mathbf{Q}_m = ( {Q}_{m,1}, {Q}_{m,2}),  \mathbf{Q} = ( {Q}_{1} , {Q}_{2}) \in \mathcal{P}_{n,1}\times \mathcal{P}_{n,2}$.
From the location of the zeros of the polynomials it readily follows that
\[ \lim_{m\to \infty}\|Q^{-1}_{m,1} -  Q^{-1}_1\|_{\Delta_2} = 0, \qquad  \lim_{m\to \infty}\|Q^{-1}_{m,2} -  Q^{-1}_2\|_{\Delta_1}=0.
\]
Therefore, for each fixed $\nu \geq 0$
\begin{equation}
	\label{moments}
	\lim_{m\to \infty}\int x^\nu \frac{\D \sigma_1}{|Q_{m.2}|} = \int x^\nu \frac{\D \sigma_1}{|Q_{2}|} =      c_{\nu}
\end{equation}
and similarly for the other sequence of polynomials. If $Q_1^*$ is the $n$-th monic orthogonal polynomial with respect to $ \frac{\D \sigma_1}{|Q_{2}|} $, the determinantal formula for the orthogonal polynomials allows to write
\[
	Q_1^*(x) = C_n^{-1} \left|
	\begin{array}{cccc}
	c_0 & c_1 & \cdots & c_{n} \\
	c_1 & c_2 & \cdots & c_{n+1} \\
	\vdots & \vdots & \ddots & \vdots\\
	c_{n-1} & c_n & \cdots & c_{2n-1} \\
	1 & x & \cdots & x^n
	\end{array}
	\right| , \qquad C_n = \left|
	\begin{array}{cccc}
	c_0 & c_1 & \cdots & c_{n-1} \\
	c_1 & c_2 & \cdots & c_{n} \\
	\vdots & \vdots & \ddots & \vdots\\
	c_{n-1} & c_n & \cdots & c_{2n-2} \\
	\end{array}
	\right| \neq 0 ,
	\]
and similar for the other $n$-th orthogonal polynomials $Q^*_{m,1}$, $Q^*_{m,2}$, $Q^*_{2} $. The determinantal formula shows that the orthogonal polynomials depend continuously on the moments of the corresponding measure. This together with \eqref{moments} clearly imply that
\[ \lim_{m\to \infty} d_n(\mathbf{Q}^*_{m},\mathbf{Q}^*) = 0.
\]

From Proposition \ref{prop1} it follows that if we take $\tilde{h}_n = |h_{n,1}|$ then
\[ \tilde{T}_n(Q_{n,1},Q_{n,2}) = (Q_{n,1},Q_{n,2}),
\]
where $(Q_{n,1},Q_{n,2})$ is the unique pair of polynomials of degree $n$ verifying \eqref{int5}-\eqref{int8}. Therefore, in
this case $(Q_{n,1},Q_{n,2})$ is a fixed point of the operator $\tilde{T}_n$. In the case of arbitrary $\tilde{h}_n$  it is not
difficult to prove that $\tilde{T}_n$ also has fixed points. (In general, it may not be unique.)

Indeed, given $n$ if $(\tilde{Q}_{n,1},\tilde{Q}_{n,2})$ is a fixed point then the $n$ zeros of $\tilde{Q}_{n,1}$ must lie in
$\Delta_1$ and the $n$ zeros of $\tilde{Q}_{n,2}$ must be in $\Delta_2$. Consequently, it is sufficient to restrict the
operator $\tilde{T}_n$ to the class $\tilde{\mathcal{P}}_{n,1}\times \tilde{\mathcal{P}}_{n,2}$ of all pairs of monic
polynomials whose first component has all its zeros on $\Delta_1$ and the second has its zeros on $\Delta_2$.
Suppose that
\[ \tilde{Q}_{n,1}(x) = \prod_{j=1}^n (x - x_{n,j}), \qquad \tilde{Q}_{n,2}(x) = \prod_{j=1}^n (x - y_{n,j}).
\]
Assume that the zeros are indexed in such a way that
\[ a_1 \leq x_{n,1} \leq \cdots \leq x_{n,n} \leq b_1, \qquad a_2 \leq y_{n,1} \leq \cdots \leq y_{n,n} \leq b_2.
\]
There is a canonical homeomorphism  between $\tilde{\mathcal{P}}_{n,1}\times \tilde{\mathcal{P}}_{n,2}$
and $\tilde{\Delta}_1 \times \tilde{\Delta}_2$, where $\tilde{\Delta}_k, k=1,2,$ is the subset of $\Delta_k^n$ made up of all
points whose coordinates are increasing, given by
\[ (\tilde{Q}_{n,1},\tilde{Q}_{n,2}) \longrightarrow ((x_{n,1},\ldots,x_{n,n}), (y_{n,1},\ldots,y_{n,n}))
\]
The operator $\tilde{T}_n$ induces an operator from $\tilde{\Delta}_1 \times \tilde{\Delta}_2$ into itself, where the image
is determined by the zeros of $(Q^*_{n,1}, Q_{n,2}^*)= \tilde{T}_n(\tilde{Q}_{n,1}, \tilde{Q}_{n,2})$. The induced operator
is continuous and $\tilde{\Delta}_1 \times \tilde{\Delta}_2$ is a convex compact subset of $\R^n \times \R^n$; therefore,
by Brouwer's fixed point theorem the induced operator has at least one fixed point. Consequently, so does $\tilde{T}_n$.

 We are ready to use Theorem \ref{main_th}.

\begin{proposition}
	\label{prop3}
	Assume that $(\mu_1,\mu_2) \in \mathcal{S}(\mathbf{\Delta})$ and for each $n \geq0$, $(\tilde{Q}_{n,1},\tilde{Q}_{n,2})$ is the pair of monic polynomials of degree $n$ which satisfies  \eqref{predet*}. Let $(\sigma_1,\sigma_2) \in \mathcal{S}(\mathbf{\Delta})$  and  let $({Q}^*_{n,1},{Q}^*_{n,2}) = \tilde{T}_n(\tilde{Q}_{n,1},\tilde{Q}_{n,2})$ where $(\tilde{h}_n)_{n\geq 0}$ fulfills \eqref{hn}. Then
 	\begin{equation}
 	\label{limfund1}
 	\lim_n \frac{q^*_{n,1}(z)}{C_1^{n}\Phi_1^{n}(z)} = \frac{1}{\sqrt{2\pi}}\mathsf{G}(f_2^{-1}  \tilde{h} \sigma_1,z), \qquad \lim_n \frac{q^*_{n,2}(z)}{C_2^{n}\Phi_2^{n}(z)} = \frac{1}{\sqrt{2\pi}}\mathsf{G}(f_{1}^{-1}   \sigma_2,z),
 	\end{equation}
	uniformly on compact subsets of $\Omega_1$ and $\Omega_2$, respectively,  $f_k =  {\mathsf{G}(\mu_k,\cdot)}/{\mathsf{G}(\mu_k,\infty)}, k=1,2$, and $q^*_{n,k} = {\kappa}^*_{n,k} Q^*_{n,k} $ is the corresponding orthonormal polynomial of degree $n$ (see \eqref{kappa*}). Additionally,
\begin{equation}
	\label{conductor1}
	\lim_n \frac{{\kappa}^*_{n,1}}{C_1^n} = \frac{1}{\sqrt{2\pi}}\mathsf{G}(f_{2}^{-1}   \tilde{h} \sigma_1,\infty), \qquad
 	\lim_n \frac{{\kappa}^*_{n,2}}{C_2^n} = \frac{1}{\sqrt{2\pi}}\mathsf{G}(f_{1}^{-1}   \sigma_2,\infty),  \qquad k=1,2.
\end{equation}

Consequently,
\begin{equation}
 	\label{limfund*1}
 	\lim_n \frac{Q^*_{n,1}(z)}{\Phi_1^{n}(z)} = \frac{\mathsf{G}(f_2^{-1}  \tilde{h} \sigma_1,z)}{ \mathsf{G}(f_2^{-1} \tilde{h} \sigma_1,\infty)}, \qquad \lim_n \frac{Q^*_{n,2}(z)}{ \Phi_2^{n}(z)} = \frac{ \mathsf{G}(f_{1}^{-1}   \sigma_2,z)}{\mathsf{G}(f_{1}^{-1} \sigma_2,\infty)}.
 	\end{equation}
\end{proposition}

 \begin{proof}
	It is easy to see that the sequences $(\tilde{h}_n \D \sigma_1, \tilde{Q}_{n,2})_{n\geq 0},  (\D \sigma_2, \tilde{Q}_{n,1})_{n\geq 0}$ verify $i)-iv)$ on the intervals $\Delta_1$ and $\Delta_2$. Therefore, the assumptions of Theorem \ref{main_th} are fulfilled. Consequently, \eqref{limfund1}-\eqref{limfund*1} follow directly from Proposition \ref{prop2} and Theorem \ref{main_th}  taking into account the equilibrium equations \eqref{vector_equil} verified by the equilibrium measures and the defining formulas \eqref{compfunc} for the functions $\Phi_k$ and constants $C_k, k=1,2$.
\end{proof}

Formulas \eqref{limfund1}-\eqref{limfund*1} describe the strong asymptotic behavior of the components of the image of $\tilde{T}_n$. In the next section, we give an operator approach to interpret the limiting functions appearing in these relations.

\subsection{The operator $T_\mathbf{w}$}

The Szeg\H{o} functions which describe the limits \eqref{limfund1}-\eqref{limfund*1} verify the boundary equations (see \eqref{Glimit}-\eqref{szego_BVP})
\begin{equation}
	\label{boundeq1}
	|\mathsf{G}(f_2^{-1}  \tilde{h} \sigma_1,x)|^2 = \frac{|f_2(x)|}{\sqrt{(b_1 - x)(x-a_1)}(\tilde{h}\sigma_1')(x)}, \quad
	\mbox{a.e. on} \quad [a_1,b_1] = \Delta_1,
\end{equation}
and
\begin{equation}
	\label{boundeq2}
	|\mathsf{G}(f_1^{-1}    \sigma_2,x)|^2 = \frac{|f_1(x)|}{\sqrt{(b_2 - x)(x-a_2)}\sigma_2'(x)}, \quad
	\mbox{a.e. on} \quad [a_2,b_2] = \Delta_2.
\end{equation}
The functions $f_1,f_2$ themselves are expressed in terms of Szeg\H{o} functions and $\tilde{h}$ is a positive continuous function on $\Delta_1$. The Szeg\H{o} functions above are symmetric with respect to the real line, never equal zero, and are positive at infinity. Consequently, on the real line, outside of the intervals supporting their defining measures, they are positive. Relations \eqref{boundeq1}-\eqref{boundeq2} suggest the definition of an operator.

Let $\mathbf{\Delta} = (\Delta_1,\Delta_2)$. We denote by $\mathbf{C}_{\mathbf{\Delta}}$ the space of all pairs $\mathbf{g} = (g_1,g_2)$ of real valued  functions such  that $g_1$ is continuous on $\Delta_2$ and $g_2$ is continuous on $\Delta_1$. The functions $g_1$ and $g_2$ could be defined on certain supersets of $\Delta_2$ and $\Delta_1$, respectively, but for the time being we only need to specify their analytic properties as indicated. Set
\[ \|\mathbf{g}\|_{\mathbf{C}_{\mathbf{\Delta}}} := \max\{\|g_1\|_{\Delta_2},\|g_2\|_{\Delta_1}\},
\]
where $\|\cdot\|_X$ denotes the sup norm on $X$. Obviously $(\mathbf{C}_{\mathbf{\Delta}},\|\cdot\|_{\mathbf{C}_{\mathbf{\Delta}}})$ is a Banach space. Consider the cone $\mathbf{C}^+_{\mathbf{\Delta}}$ of all the vectors in $\mathbf{C}_{\mathbf{\Delta}}$ such that $g_1$ is positive on $\Delta_2$ and $g_2$ is positive on $\Delta_1$. The application $(g_1,g_2) \mapsto (\ln g_1, \ln g_2)$ establishes a homeomorphism between $\mathbf{C}^+_{\mathbf{\Delta}}$ and $\mathbf{C}_{\mathbf{\Delta}}$. Given $\mathbf{g}^{(1)} = (g^{(1)}_1,g^{(1)}_2),\mathbf{g}^{(2)} = (g^{(2)}_1,g^{(2)}_2) \in \mathbf{C}^+_{\mathbf{\Delta}}$, set
\[ d(\mathbf{g}^{(1)}, \mathbf{g}^{(2)})  := \max\{\|\ln(g^{(1)}_1/ {g}^{(2)}_1)\|_{\Delta_2}, \|\ln(g^{(1)}_2/{g}^{(2)}_2)\|_{\Delta_1}\}.
\]
It is easy to check that $(\mathbf{C}^+_{\mathbf{\Delta}},d)$ is a complete metric space. Certainly, on
$\mathbf{C}^+_{\mathbf{\Delta}}$ we can also consider the norm $\|\cdot\|_{\mathbf{C}_{\mathbf{\Delta}}}$ but $\mathbf{C}^+_{\mathbf{\Delta}}$ is not complete with that norm; however, given a sequence $(\mathbf{g}^{(n)})_{n\geq 0}\subset \mathbf{C}^+_{\mathbf{\Delta}}$ and $\mathbf{g} \in \mathbf{C}^+_{\mathbf{\Delta}}$, we have
	\begin{equation}
		\label{equiv}
		\lim_n \|\mathbf{g}^{(n)} - \mathbf{g}\|_{\mathbf{C}_{\mathbf{\Delta}}} = 0 \quad \Leftrightarrow \quad \lim_n
d(\mathbf{g}^{(n)},\mathbf{g}) =0.
	\end{equation}

	Now we are in a position to give a precise definition of the operator hinted by relations
\eqref{boundeq1}-\eqref{boundeq2}.
	\begin{definition}
		\label{oper_T}
		Let $w_1$ and $w_2$ be two integrable functions satisfying Szeg\H{o}'s condition on $\Delta_1$ and $\Delta_2$,
respectively, and write $\mathbf{w}=(w_1,w_2)$. Define the operator
		\[ T_\mathbf{w}: \mathbf{C}^+_{\mathbf{\Delta}}\longrightarrow  \mathbf{C}^+_{\mathbf{\Delta}},
		\]
		where $T_\mathbf{w}(g_1,g_2) = (g_1^*,g_2^*)$ is the pair of Szeg\H{o} functions,  $g_k^*\in
\mathcal{H}(\Omega_k), k=1,2,$ verifying
		\begin{align*}
			|g_1^*(x)|^2 =& \frac{g_2(x)}{w_1(x)}, \quad
			\mbox{a.e. on} \quad [a_1,b_1] = \Delta_1,\\ 
			|g_2^*(x)|^2 =& \frac{g_1(x)}{w_2(x)}, \quad
			\mbox{a.e. on} \quad [a_2,b_2] = \Delta_2. 
		\end{align*}	
	\end{definition}

	From the definition of the Szeg\H{o} function it readily follows that $g_1^*$ is positive and continuous on $\R \setminus \Delta_1 \supset \Delta_2$ and $g_2^*$ is positive and continuous on $\R \setminus \Delta_2 \supset \Delta_1$.
	
	Finding $g_k^*, k=1,2,$ reduces to solving the Dirichlet problems for a harmonic function $u_k$ in $\Omega_k,$ with boundary values integrable on $\Delta_k$ and equal to $\frac{1}{2} \ln (g_2/w_1)$ a.e. on $\Delta_1$ in the case of $u_1$  and $\frac{1}{2} \ln (g_1/w_2)$ a.e. on $\Delta_2$ in the case of $u_2$, and the subsequent problem of finding their harmonic conjugates $\widetilde{u}_k,\widetilde{u}_k(\infty)=0,$ in $\Omega_k$. Then
	\[ g_k^* = \exp(u_k +i\, \widetilde{u}_k), \qquad k=1,2.
	\]

Set $\mathbf{\Omega} = (\Omega_1,\Omega_2)$. Let $\mathbf{h}_{\mathbf{\Omega}}$ be the set of pairs of harmonic functions in $\Omega_1$ and $\Omega_2$, respectively, with integrable boundary values. Given $\mathbf{g} = (g_1,g_2) \in \mathbf{C}^+_{\mathbf{\Delta}}$ let $\mathbf{\chi}= (\chi_1,\chi_2) = (\ln g_1,\ln g_2) \in \mathbf{C}_{\mathbf{\Delta}}$.
The map $T_\mathbf{w}$ induces the map
\[ t_\mathbf{w}: \mathbf{C}_{\mathbf{\Delta}} \longrightarrow \mathbf{h}_{\mathbf{\Omega}}\subset
\mathbf{C}_{\mathbf{\Delta}},
\]
where
\[t_\mathbf{w}(\chi) :=  \frac{1}{2} (P(\chi) + \beta), \qquad \chi = (\chi_1,\chi_2)^t \in \mathbf{C}_{\mathbf{\Delta}},
\]
$\beta = (\beta_1,\beta_2)^t$ is the (column) vector made up of harmonic functions with boundary values
\[ \beta_k(x) = -\ln w_k(x), \qquad \mbox{a.e. on} \quad \Delta_k, \,k=1,2
\]
and $P$ is the linear operator
\[ P :=
\left(
\begin{array}{cc}
0 & P_{1,2} \\
P_{2,1} & 0
\end{array}
\right),
\]
such that $P_{1,2}(\chi_2)$ is the harmonic function on $\Omega_1$ with boundary values equal to $\chi_2$ on
$\Delta_1$, and $P_{2,1}(\chi_1)$ is the harmonic function on $\Omega_2$ with boundary values equal to $\chi_1$ on
$\Delta_2$.

The following result is contained in \cite[Proposition 1.1]{sasha}.

\begin{proposition}
	\label{prop4}
	The operator $T_\mathbf{w}$ (see Def. \ref{oper_T}) is a contraction in $\mathbf{C}^+_{\mathbf{\Delta}}$ with respect
to the metric $d$. More precisely
		\[ d(T_\mathbf{w}(\mathbf{g}^{(1)}),T_\mathbf{w}(\mathbf{g}^{(2)}))  \leq
\frac{1}{2}d(\mathbf{g}^{(1)},\mathbf{g}^{(2)}), \qquad \mathbf{g}^{(1)},\mathbf{g}^{(2)} \in
\mathbf{C}^+_{\mathbf{\Delta}}.
		\]
	Therefore, the map $T_\mathbf{w}$ has a unique fixed point in $\mathbf{C}^+_{\mathbf{\Delta}}$.
\end{proposition}

\begin{proof}
	The proof is  simple so, for completeness, we include it. As mentioned above, $(\mathbf{C}^+_{\mathbf{\Delta}},d)$ is a complete metric space  so the second statement follows from the first.

	Set $\mathbf{\chi}^{(k)}  := (\ln g^{(k)}_1,\ln g^{(k)}_2), k=1,2$.  From the definitions of $d, T_\mathbf{w},$ and
$t_\mathbf{w}$  it follows that
		\[ d(T_\mathbf{w}(\mathbf{g}^{(1)}),T_\mathbf{w}(\mathbf{g}^{(2)})) = \|t_\mathbf{w}(\chi^{(1)}) -
t_\mathbf{w}(\chi^{(2)})\|_{\mathbb{C}_{\mathbf{\Delta}}} = \frac{1}{2}\|P(\chi^{(1)}) -
P(\chi^{(2)})\|_{\mathbb{C}_{\mathbf{\Delta}}} \leq
		\]
		\[ \frac{1}{2} \|P\|\|\chi^{(1)} - \chi^{(2)}\|_{\mathbb{C}_{\mathbf{\Delta}}} = \frac{1}{2}
d(\mathbf{g}^{(1)},\mathbf{g}^{(2)}),
		\]
	where in the last inequality the maximum principle is used to establish that $\|P\| = 1$.
\end{proof}

 Consider now the particular case $T_{\widetilde{\mathbf{w}}}$ where
\begin{equation}
	\label{oper_T_tilde}
	\widetilde{\mathbf{w}} := \left(\sqrt{(b_1 - x)(x-a_1)}(\tilde{h}\sigma_1')(x), \;\sqrt{(b_2 - x)(x-a_2)} \sigma_2'(x)\right).
\end{equation}
Here $\tilde{h}, \sigma_1, \sigma_2$ are the same as in \eqref{hn}, \eqref{boundeq1} and \eqref{boundeq2}. Then, let $\mathbf{\mathsf{G}} = (\mathsf{G}_1, \mathsf{G}_2)$ be the unique fixed point of the operator $T_{\widetilde{\mathbf{w}}}$.
That is
	\[T_{\widetilde{\mathbf{w}}}(\mathbf{\mathsf{G}}) = \mathbf{\mathsf{G}} \]
and the components of $\mathbf{\mathsf{G}} $ are characterized by the  system of boundary values
\begin{equation*}
	\label{boundeq5}
	|\mathsf{G}_1(x)|^2 = \frac{\mathsf{G}_2(x)}{\sqrt{(b_1 - x)(x-a_1)}( \tilde{h}\sigma_1')(x)}, \quad
	\mbox{a.e. on} \quad [a_1,b_1] = \Delta_1,
\end{equation*}
and
\begin{equation*}
	\label{boundeq6}
	|\mathsf{G}_2(x)|^2 = \frac{\mathsf{G}_1(x)}{\sqrt{(b_2 - x)(x-a_2)} \sigma_2'(x)}, \quad
	\mbox{a.e. on} \quad [a_2,b_2] = \Delta_2.
\end{equation*}
Obviously, the components of $\mathbf{\mathsf{G}}$ are Szeg\H{o} functions in $\Omega_1$ and $\Omega_2$,
respectively.

Now, we must show  that any neighborhood of a fixed point of the operator $T_{\widetilde{\mathbf{w}}}$ determines fixed points of the operators $\tilde{T}_n$ for all sufficiently large $n$. By Proposition \ref{prop1}, when we take $\tilde{h}_n = |h_{n,1}|$ as in \eqref{Qs}, the operator $\tilde{T}_n$ has only one fixed point. We need one last ingredient.

\begin{proposition}
	\label{prop5}
	Let $(\tilde{Q}_{n,1},\tilde{Q}_{n,2})_{n\geq 0}$ be an arbitrary sequence of vector polynomials such that
$(\tilde{Q}_{n,1},\tilde{Q}_{n,2}) \in \mathcal{P}_{n,1} \times \mathcal{P}_{n,2}$. Set
	\[ \mathbf{f}^{(n)} = \left(\frac{\tilde{Q}_{n,1}}{\Phi_1^{n}} , \frac{\tilde{Q}_{n,2}}{\Phi_2^{n}}\right).
	\]
	Assume that there exists $\mathbf{f} = (f_1,f_2) \in \mathbf{C}^+_{\mathbf{\Delta}}$ and a sequence of non-negative
integers $\Lambda$ such that
	\begin{equation}
		\label{limden}
		\lim_{n\in \Lambda} \|\mathbf{f}^{(n)}-\mathbf{f}\|_{\mathbf{C}_{\mathbf{\Delta}}}  = 0.
 	\end{equation}
 	Let $(\sigma_1,\sigma_2) \in \mathcal{S}(\mathbf{\Delta})$  and  let $({Q}^*_{n,1},{Q}^*_{n,2}) =
 \tilde{T}_n(\tilde{Q}_{n,1},\tilde{Q}_{n,2})$ where $(\tilde{h}_n)_{n\geq 0}$ fulfills \eqref{hn}. Then
 	\begin{equation}
 		\label{limfund}
 		\lim_{n \in \Lambda} \frac{q^*_{n,1}(z)}{C_1^{n}\Phi_1^{n}(z)} = \frac{1}{\sqrt{2\pi}}\mathsf{G}(f_2^{-1}  \tilde{h}
 \sigma_1,z), \qquad \lim_{n \in \Lambda} \frac{q^*_{n,2}(z)}{C_2^{n}\Phi_2^{n}(z)} =
 \frac{1}{\sqrt{2\pi}}\mathsf{G}(f_{1}^{-1}   \sigma_2,z),
 	\end{equation}
 	uniformly on compact subsets of $\Omega_1$ and $\Omega_2$, respectively. Additionally,
	\begin{equation*}
 		\lim_{n\in \Lambda} \frac{{\kappa}^*_{n,1}}{C_1^n} = \frac{1}{\sqrt{2\pi}}\mathsf{G}(f_{2}^{-1}   \tilde{h}
 \sigma_1,\infty), \qquad
		\lim_{n\in \Lambda}\frac{{\kappa}^*_{n,2}}{C_2^n} = \frac{1}{\sqrt{2\pi}}\mathsf{G}(f_{1}^{-1}   \sigma_2,\infty),
\qquad k=1,2.
 	\end{equation*}
 Consequently,
 	\begin{equation}
 		\label{limfund*}
 		\lim_{n\in \Lambda} \frac{Q^*_{n,1}(z)}{\Phi_1^{n}(z)} =  \frac{\mathsf{G}(f_2^{-1}  \tilde{h} \sigma_1,z)}{
 \mathsf{G}(f_2^{-1}  \tilde{h} \sigma_1,\infty)}, \qquad \lim_{n\in \Lambda} \frac{Q^*_{n,2}(z)}{ \Phi_2^{n}(z)} = \frac{
 \mathsf{G}(f_{1}^{-1}   \sigma_2,z)}{ \mathsf{G}(f_{1}^{-1}   \sigma_2,\infty)}.
 	\end{equation}
 \end{proposition}

 \begin{proof}
	The proof is identical to that of Proposition \ref{prop3}. In that proof, it is not used that the full sequences of indices is considered and the result only depends on the asymptotic behavior of the sequences of denominators of the varying part of the measures of orthogonality on the intervals $\Delta_1$ and $\Delta_2$, respectively, for which assumption \eqref{limden} was included. The details are left to the reader.
 \end{proof}

\subsection{Proof of Theorem \ref{main}}

Let ${\mathsf{G}}= (\mathsf{G}_1,\mathsf{G}_2)$ be the fixed point of the operator $T_{\widetilde{\mathbf{w}}}$ (see Def. \ref{oper_T} and \eqref{oper_T_tilde}). The function $ \mathsf{G}_k, k=1,2$ is a Szeg\H{o} function in $\Omega_k$; therefore, the value $\mathsf{G}_k(\infty)$ is well defined. Set
	\[ H^{+,n} := \left\{\left(\frac{\mathsf{G}_1(\infty)P_{n,1}}{\Phi_1^n}, \frac{\mathsf{G}_2(\infty)P_{n,2}}{\Phi_2^n}\right): P_{n,k} \in \mathcal{P}_{n,k}, k=1,2\right\}.
	\]
(Recall that $\mathcal{P}_{n,k}$ is the set of all monic polynomials of degree $n$ with real coefficients whose zeros lie in $\Omega_j, j \neq k, j,k =1,2$.)

Let $\tilde{T}_{n,1}$ and $\tilde{T}_{n,2}$ be the operators defined on $\mathcal{P}_{n,1}\times \mathcal{P}_{n,2}$ which determine the components of $\tilde{T}_n$; that is, $\tilde{T}_n = (\tilde{T}_{n,1},\tilde{T}_{n,2})$ (see \eqref{Tn}). Define
	\[  T_n : H^{+,n} \longrightarrow H^{+,n}
	\]
where
	\[ T_n\left(\frac{\mathsf{G}_1(\infty)P_{n,1}}{\Phi_1^n}, \frac{\mathsf{G}_2(\infty)P_{n,2}}{\Phi_2^n}\right) = \left(
\frac{\mathsf{G}_1(\infty)\tilde{T}_{n,1}(P_{n,1},P_{n,2})}{\Phi_1^n},
\frac{\mathsf{G}_1(\infty)\tilde{T}_{n,2}(P_{n,1},P_{n,2})}{ \Phi_2^n}\right).
	\]
Notice that any fixed point of $T_n$ generates a fixed point of $\tilde{T}_n$. The continuity of $\tilde{T}_n$ implies the
continuity of $T_n$.

\begin{theorem}
	\label{main}
	Let $(\sigma_1,\sigma_2) \in \mathcal{S}(\mathbf{\Delta})$  and  $(\tilde{h}_n)_{n\geq 0}$ fulfills \eqref{hn}. Then, there exists a sequence $(Q_{n,1},Q_{n,2})_{n\geq 0}$, where $(Q_{n,1},Q_{n,2})$ is a fixed point of $\tilde{T}_n$, such that
	\begin{equation}
 		\label{limfund***}
 		\lim_n \frac{Q_{n,k}(z)}{\Phi_k^{n}(z)} =  \frac{\mathsf{G}_k(z)}{ \mathsf{G}_k(\infty)} , \qquad k=1,2,
 	\end{equation}
	uniformly on compact subsets of $\Omega_k$, where $\mathsf{G} = (\mathsf{G}_1, \mathsf{G}_2)$ is the fixed point of $T_{\widetilde{\mathbf{w}}}$ and $\Phi_k$ are as in \eqref{compfunc}. Additionally, if
 		\[ \  (\kappa_{n,2})^{-2} := \int    Q^2_{n,2}(x)  \frac{{\rm d}\sigma_{2}(x)}{|{Q}_{n,1}(x)|}, \qquad (\kappa_{n,1})^{-2}
 :=  \int  Q^2_{n,1}(x)      \frac{\tilde{h}_{n}(x){\rm d}\sigma_{1}(x)}{|{Q}_{n,2}(x)|},
 		\]
 	and $q_{n,k} = \kappa_{n,k}Q_{n,k}, k=1,2$, then
	\begin{equation*}
 		\lim_n \frac{{\kappa}_{n,k}}{C_k^n} = \frac{1}{\sqrt{2\pi}}\mathsf{G}_k(\infty),
 	\end{equation*}
 	and
 	\begin{equation}
 		\label{limfund2}
 		\lim_n \frac{q_{n,k}(z)}{C_k^{n}\Phi_k^{n}(z)} = \frac{1}{\sqrt{2\pi}}\mathsf{G}_k(z), \qquad k=1,2,
 	\end{equation}
 	uniformly on compact subsets of $\Omega_k$.
\end{theorem}

\begin{proof}
Due to the way in which $H^{+,n}, T_n,$ and $T_{\widetilde{\mathbf{w}}}$ were defined, the statements
\eqref{limfund***}-\eqref{limfund2} follow directly from Proposition \ref{prop5} if we show that there exists a sequence
$(\mathbf{g}^{(n)})_{n \geq n_0}$, where $\mathbf{g}^{(n)}$ is a fixed point of $T_n$, such that
	\begin{equation}
		\label{fixed} \lim_{n \geq n_0} \|\mathbf{g}^{(n)} - \mathsf{G}\|_{\mathbf{C}_{\mathbf{\Delta}}} = 0.
	\end{equation}

	The components of $\mathsf{G}$ are Szeg\H{o} functions in $\Omega_k, k=1,2$, respectively.  Let $\mathbf{K} = (K_1, K_2)$ be a pair of  non intersecting closed disks, symmetric with respect to $\R$,  whose interior in the Euclidean topology of $\C$ verify
		\[ \Delta_2 \subset \stackrel{\circ}{K}_1, \qquad  \Delta_1 \subset \stackrel{\circ}{K}_2.
		\]
	By $H^+(\mathbf{K})$ we denote the cone of all pairs $(g_1,g_2)$ of functions such that $g_k$ is holomorphic and different from zero in $\stackrel{\circ}{K}_k$, and positive on $\stackrel{\circ}{K}_k \cap \R$. For $\mathbf{g} = (g_1,g_2) \in H^+(\mathbf{K})$ we define
		\[ \|\mathbf{g}\|_{\mathbf{K}} := \max\{\{\sup|g_k(z)|: z \in \stackrel{\circ}{K}_k\}: k=1,2\} \,(\leq \infty),
		\]
and
\[ \min_{\mathbf{\Delta}} \mathbf{g} := \min \{\min_{x\in \Delta_2} g_1(x), \min_{x \in \Delta_1} g_2(x)\}.
\]
Fix a constant $C > 0$. Define
		\[ H^+(\mathbf{K}, C) := \{\mathbf{g}  \in H^+(\mathbf{K}) : \|\mathbf{g}\|_{\mathbf{K}} \leq C, \min_{\mathbf{\Delta}}
\mathbf{g} \geq   C^{-1}\}.
\]
Take $C$ sufficiently large so that $\mathsf{G} \in H^+(\mathbf{K}, C)$.	

Let $\mathbf{g}^{(1)} , \mathbf{g}^{(2)}\in  H^+(\mathbf{K}, C)$ and let $0\leq \beta \leq 1$,
then
\[ \|\beta \mathbf{g}^{(1)} + (1-\beta) \mathbf{g}^{(2)}\|_{\mathbf{K}} \leq \beta \|\mathbf{g}^{(1)}\|_{\mathbf{K}} +
(1-\beta)\|\mathbf{g}^{(2)}\|_{\mathbf{K}} \leq C,
\]
and
\[ \min_{ \mathbf{\Delta}} \left(\beta \mathbf{g}^{(1)} + (1-\beta) \mathbf{g}^{(2)}\right)   \geq \beta \min_{\mathbf{\Delta}}
\mathbf{g}^{(1)} + (1-\beta)\min_{   \mathbf{\Delta}} \mathbf{g}^{(2)} \geq C^{-1},
\]
Therefore, $ \beta \mathbf{g}^{(1)} + (1-\beta) \mathbf{g}^{(2)}\in H^+(\mathbf{K}, C)$.
This shows that $H^+(\mathbf{K}, C)$ is convex.

On the other hand, if $(\mathbf{g}^{(n)})_{n \geq 0}$ is an arbitrary sequence of elements in $ H^+(\mathbf{K}, C)$. Then the components form normal families in $\stackrel{\circ}{K}_1$ and $\stackrel{\circ}{K}_2$, respectively. Therefore, there exists a sequence of indices $\Lambda$ such that $(\mathbf{g}^{(n)})_{n \in \Lambda}$ converges componentwise to some vector function $\mathbf{g}$ uniformly on each compact subset of $\stackrel{\circ}{K}_1$ and $\stackrel{\circ}{K}_2$, respectively. The components of $\mathbf{g}$ are, therefore, analytic on $\stackrel{\circ}{K}_1$ and $\stackrel{\circ}{K}_2$, respectively. The uniform limit of holomorphic functions which never equal zero must either be identically equal to zero or never zero. The first case is not possible because
\[ \min_{  \mathbf{\Delta}} \mathbf{g} = \lim_{n \in \Lambda} \min_{  \mathbf{\Delta}} \mathbf{g}^{(n)} \geq C^{-1}.
\]
Also
\[ \|\mathbf{g}\|_{\mathbf{K}} = \lim_{n\in \Lambda} \|\mathbf{g}^{(n)}\|_{\mathbf{K}} \leq C.
\]
Consequently, $\mathbf{g} \in H^+(\mathbf{K},C)$. We conclude that $H^+(\mathbf{K},C)$ is compact.

Fix an arbitrary $\theta > 0$.   Let
		\[ \omega(\theta)= \{\mathbf{g} \in H^+(\mathbf{K}, C): \|\mathbf{g} -
\mathbf{\mathsf{G}}\|_{\mathbf{C}_\mathbf{\Delta}}\leq \theta\}.
		\]
	This is a closed subset of $H^+(\mathbf{K}, C)$ and, therefore, it is compact. Obviously, it is convex. Analogously, for
every $\varepsilon > 0$, set
		\[ \omega_\varepsilon := \{ \mathbf{g}\in H^+(\mathbf{K}, C) : d(\mathbf{g},\mathbf{\mathsf{G}}) \leq \varepsilon\}.
		\]
	There exists $\varepsilon_0$ such that
		\[ \omega_{\varepsilon} \subset \omega(\theta), \qquad 0 < \varepsilon \leq \varepsilon_0,
		\]
	for, otherwise, we could find a sequence of vector functions in $H^+(\mathbf{K}, C) \subset
\mathbf{C}_{\mathbf{\Delta}}^+$ which converges to $\mathbf{\mathsf{G}}$  in the $d$ metric but not in the
$\|\cdot\|_{\mathbf{C}_\mathbf{\Delta}}$ norm which would contradict \eqref{equiv}.

	Take
		\[ \omega_{\varepsilon,n} =  \omega_{\varepsilon}  \cap H^{+,n}.
		\]
	For each  fixed $n$ the set $\omega_{\varepsilon,n}$ is a closed, bounded subset of a finite dimensional space,
therefore it is compact.

	Let $\mu_k$ be the representing measure of $\mathsf{G}_k$ so that $\mathsf{G}_k(z) = \mathsf{G}(\mu_k, z), z \in
\Omega_k.$ From Proposition \ref{prop2}, using  \eqref{predet} and \eqref{equiv} it follows that
		\[ \lim_n d(\mathbf{g}^{(n)}, \mathbf{\mathsf{G}}) = 0,
		\]
	where
		\[ \mathbf{g}^{(n)} := \left( \frac{\mathsf{G}_1(\infty)\tilde{Q}_{n,1}}{\Phi_1^n},
\frac{\mathsf{G}_2(\infty)\tilde{Q}_{n,2}}{\Phi_2^n}\right)
		\]
	and $\tilde{Q}_{n,k}, k=1,2,$ is given by \eqref{poltilde}. Consequently, for every fixed $\varepsilon, 0 < \varepsilon \leq \varepsilon_0$, there exists $n_0$ such that $\omega_{\varepsilon,n} \neq \varnothing$ for all $n \geq n_0$. Using the structure of the elements of $H^{+,n}$, the definition of the metric $d$, and the monotonicity of the logarithm, it is easy to verify that $\omega_{\varepsilon,n}$ is convex.

	Let us show that $T_n(\omega_{\varepsilon,n}) \subset \omega_{\varepsilon,n}$ for all sufficiently large $n$. We claim that there exists $n_0$ such that for all $n \geq n_0 $ and $\mathbf{g} \in \omega_{\varepsilon, n}$, we have
	\begin{equation}
		\label{desig}
		d(T_n({\mathbf{g}}),T_{\widetilde{\mathbf{w}}}(\mathbf{g})) < \varepsilon/2.
	\end{equation}
	Should this not occur, we could find a sequence $(\mathbf{g}^{(n)})_{n \in \Lambda}, \mathbf{g}^{(n)} \in \omega_{\varepsilon,n},$ such that
	\begin{equation}
		\label{contra}
		d(T_n({\mathbf{g}}^{(n)}),T_{\widetilde{\mathbf{w}}}(\mathbf{g}^{(n)})) \geq \varepsilon/2
	\end{equation}
	The elements of $\omega_{\varepsilon,n}$ belong to $H^+(K,C)$; therefore, $(\mathbf{g}^{(n)})_{n \in \Lambda}$ is uniformly bounded in the $\|\cdot\|_{\mathbf{K}}$ norm.  Consequently, there exists $\mathbf{g} \in H^+(K,C)$ and a subsequence of indices $\Lambda' \subset \Lambda$ such that
		\[ \lim_{n\in \Lambda'} \|{\mathbf{g}}^{(n)} - \mathbf{g}\|_{\mathbf{K}} = 0.
		\]
	In particular,
		\[ \lim_{n \in \Lambda'} \|{\mathbf{g}}^{(n)} - \mathbf{g}\|_{\mathbf{C}_{\mathbf{\Delta}}} = 0.
		\]
	Then, according to \eqref{limfund*} in Proposition \ref{prop5}
		\[ \lim_{n \in \Lambda'} \|T_n({\mathbf{g}}^{(n)}) - T_{\widetilde{\mathbf{w}}}(\mathbf{g})\|_{\mathbf{C}_{\mathbf{\Delta}}} = 0
		\]
	which contradicts \eqref{contra} due to \eqref{equiv}.

	Using the triangle inequality and \eqref{desig} for every $n \geq n_0$ and $\mathbf{g} \in \omega_{\varepsilon,n}$
		\[ d(\mathsf{G}, T_n(\mathbf{g})) \leq d(T_{\widetilde{\mathbf{w}}}(\mathsf{G}), T_{\widetilde{\mathbf{w}}}(\mathbf{g})) + d(T_{\widetilde{\mathbf{w}}}(\mathsf{g}), T_n(\mathbf{g})) < \frac{1}{2}d(\mathsf{G}, \mathbf{g}) + \frac{1}{2}\varepsilon \leq \varepsilon.
		\]
	Consequently, $T_n(\omega_{\varepsilon,n}) \subset \omega_{\varepsilon,n}$ as claimed. Now, using Brouwer's fixed point Theorem we obtain  that for all $n \geq n_0$ the operator $T_n$ has a fixed point in $\omega_{\varepsilon,n}$.

	Since $\theta > 0$ is arbitrary, we have shown that \eqref{fixed} is true and we conclude the proof.
\end{proof}
To apply Theorem \ref{main} to the case of ML Hermite-Pad\'e polynomials we need to  select  $\tilde{h}_n = |h_{n,1}|$, $\tilde{h} = h$ (see \eqref{Qs} and \eqref{int12}), and
\begin{equation}
	\label{oper_TQ}
	\tilde{\mathbf{w}} = \mathbf{w}_Q := (\sqrt{(b_1-x)(x-a_1)}(h\sigma_1')(x), \; \sqrt{(b_2-x)(x-a_2)}\sigma_2'(x)),
\end{equation}
to determine the operator $T_{\mathbf{w}_Q}$ (see Def. \ref{oper_T}).

\begin{theorem} \label{ML} Let $(\sigma_1,\sigma_2) \in \mathcal{S}(\mathbf{\Delta}) $. Let
$(Q_{n,1},Q_{n,2})_{n\geq0}$ be the sequence of ML Hermite-Pad\'e polynomials defined by  \eqref{int3}-\eqref{hn1} and let $\kappa_{n,1},\kappa_{n,2}, q_{n,1}, q_{n,2},$ and $h_{n,1}$ be defined as in \eqref{kappa}-\eqref{Qs}. Finally, let $\mathsf{G} = (\mathsf{G}_1,\mathsf{G}_2)$ be the fixed point of the operator $T_{\mathbf{w}_Q}$, as in Definition \ref{oper_T} and \eqref{oper_TQ}. Then
\begin{equation*}
 	\lim_n \frac{Q_{n,k}(z)}{\Phi_k^{n}(z)} =  \frac{\mathsf{G}_k(z)}{ \mathsf{G}_k(\infty)} , \qquad k=1,2,
 	\end{equation*}
 uniformly on compact subsets of $\Omega_k$.   Additionally,
 \begin{equation} \label{conductor*}
 \lim_n \frac{{\kappa}_{n,k}}{C_k^n} = \frac{1}{\sqrt{2\pi}}\mathsf{G}_k(\infty),
 \end{equation}
 and
 \begin{equation}
 	\label{limfund**}
 	\lim_n \frac{q_{n,k}(z)}{C_k^{n}\Phi_k^{n}(z)} = \frac{1}{\sqrt{2\pi}}\mathsf{G}_k(z), \qquad k=1,2,
 	\end{equation}
 uniformly on compact subsets of $\Omega_k$.
\end{theorem}

\begin{proof}
	It is sufficient to apply Theorem \ref{main} with $\tilde{h}_n = |h_{n,1}|$, taking note that \eqref{int12} takes place and
that according to Proposition \ref{prop1} the operator $\tilde{T}_n$ has a unique fixed point in all of
$\mathcal{P}_{n,1}\times\mathcal{P}_{n,2}$ which coincides with $(Q_{n,1},Q_{n,2})$.
\end{proof}

Our method differs from Aptekarev's in two aspects. Proposition \ref{prop2}, which plays a key role, is derived using arguments from complex function theory. The corresponding result in \cite[Theorem 2]{sasha} uses a quite intricate approximative construction on a Riemann surface. Secondly, in \cite{sasha},  Widom's approach  introduced in \cite{widom} is followed closely to obtain $L_2$ estimates, on segments of the real line, of the asymptotic behavior of the multiple orthogonal polynomials. Thus, the results are obtained for measures in the Szeg\H{o} class which are absolutely continuous with respect to the Lebesgue measure. We use instead the results obtained in Section 2 on orthogonal polynomials with respect to varying measures and do not need to restrict to absolutely continuous measures. In consequence, we only give the asymptotic in the complement of the intervals.

\subsection{Proof of Theorem \ref{mainbiort}} Recall that the biorthogonal polynomial $Q_n$ coincides
with $Q_{n,2}$. Consequently, the second relation in \eqref{asinbio} follows directly from \eqref{limfund*}.

To obtain the asymptotic of the biorthogonal polynomials $P_n$ we need a result similar to Theorem \ref{ML} working
with the definition \eqref{JLS1*}-\eqref{JLS2*} corresponding to the Nikishin system $\mathcal{N}(\sigma_2,\sigma_1)$.
We outline the main ingredients.

From  Lemma \ref{l2} and Proposition \ref{prop1} it follows that there exist a unique pair $(P_{n,1}, P_{n,2})$ of monic
polynomials of degree $n$, where $P_{n,2} = b_{n,2} = P_n$, such that
\begin{align}
	\int x^\nu P_{n,2}(x)\frac{\D\sigma_1(x)}{P_{n,1}(x)}=0,\qquad \nu=0,1,\ldots,n-1, \label{orth_P2}\\
	\int x^\nu P_{n,1}(x)\frac{\mathcal{L}_{n,1}(x)\D\sigma_2(x)}{P_{n,2}(x)}=0,\qquad \nu=0,1,\ldots,n-1,\label{orth_P1}
\end{align}	
where
\begin{equation*}
	\label{J_n1}
	\mathcal{L}_{n,1}(z):=\frac{\mathcal{B}_{n,1}(z)P_{n,2}(z)}{P_{n,1}(z)} = P_{n,2}(z)\int\frac{P_{n,2}(x)}{z-x}
\frac{\D\sigma_1(x)}{P_{n,1}(x)} =   \int\frac{P_{n,2}^2(x)}{z-x}\frac{\D\sigma_1(x)}{P_{n,1}(x)}.
\end{equation*}	
The normalization in this case is
\begin{equation*}
	\xi_{n,2}^{-2} = \int P_{n,2}^2(x)\frac{\D\sigma_1(x)}{|P_{n,1}(x)|}, \qquad
    (\xi_{n,1}\xi_{n,2})^{-2} = \int P_{n,1}^2(x)\frac{|\mathcal{L}_{n,1}(x)|\D\sigma_2(x)}{|P_{n,2}(x)|}.
\end{equation*}
Take,
\begin{equation*}
	\label{norm_J}
	p_{n,1}=\xi_{n,1}P_{n,1},\qquad p_{n,2}=\xi_{n,2}P_{n,2},\qquad\ell_{n,1}=\xi_{n,2}^2\mathcal{L}_{n,1}.
\end{equation*}	
Then, the orthogonality relation \eqref{orth_P2} and \eqref{orth_P1} can be restated as
\begin{align*}
	\int x^\nu p_{n,2}(x)\frac{\D\sigma_1(x)}{|P_{n,1}(x)|}=0,\qquad \nu=0,1,\ldots,n-1, \\
	\int x^\nu p_{n,1}(x)\frac{|\ell_{n,1}(x)|\D\sigma_2(x)}{|P_{n,2}(x)|}=0,\qquad \nu=0,1,\ldots,n-1,
\end{align*}
and the polynomials $p_{n,2}$ and $p_{n,1}$ are orthonormal with respect to the corresponding varying measures.
Following the same arguments that led us to \eqref{int12}, we obtain
\begin{equation} \label{ell}
	\lim_n |\ell_{n,1}(t)| = \frac{1}{\sqrt{|t-a_1||t-b_1|}} =: \ell(t)
\end{equation}
uniformly for $t\in\Delta_2$.

The operator $T_{\mathbf{w}_P}: \mathbf{C}^+_{\mathbf{\Delta}}\longrightarrow  \mathbf{C}^+_{\mathbf{\Delta}}$ which is relevant to describe the strong asymptotic of the polynomials $P_{n,2},P_{n,1}$ and their orthonormal versions $p_{n,2},p_{n,1}$ is the one determined by
\begin{equation}
	\label{oper_TP}
	\mathbf{w}_P := (\sqrt{(b_1-x)(x-a_1)}\sigma_1'(x), \; \sqrt{(b_2-x)(x-a_2)}(\ell\sigma_2')(x)).
\end{equation}	
In other words,  $T_{\mathbf{w}_P}(g_1,g_2) = (g_1^*,g_2^*)$ is the pair of Szeg\H{o}
functions, $g_k^*\in \mathcal{H}(\Omega_k), k=1,2,$ verifying
\begin{equation}\label{boundeq3*}
|g_1^*(x)|^2 = \frac{g_2(x)}{\sqrt{(b_1 - x)(x-a_1)} \sigma_1'(x)}, \quad
\mbox{a.e. on} \quad [a_1,b_1] = \Delta_1,
\end{equation}
and
\begin{equation}\label{boundeq4*}
|g_2^*(x)|^2 = \frac{g_1(x)}{\sqrt{(b_2 - x)(x-a_2)} (\ell\sigma_2')(x)}, \quad
\mbox{a.e. on} \quad [a_2,b_2] = \Delta_2,
\end{equation}
where $\ell$ is given in \eqref{ell}.

Following the same reasonings as before, if $(\mathsf{G}^*_1,\mathsf{G}^*_2)$ is the fixed point of the operator
$T_{\mathbf{w}_P}$ defined through \eqref{boundeq3*}-\eqref{boundeq4*}, we have
	\begin{equation}
		\label{lim_P*}
		\lim_n \frac{P_{n,1}(z)}{\Phi_{2}^n(z)} = \frac{\mathsf{G}^*_2(z)}{\mathsf{G}^*_2(\infty)},\qquad  \lim_n
\frac{P_{n,2}(z)}{\Phi_{1}^n(z)} = \frac{\mathsf{G}^*_1(z)}{\mathsf{G}^*_1(\infty)}
	\end{equation}
	uniformly on compact subsets of $\Omega_2$ and $\Omega_1$, respectively. Since $P_{n} = P_{n,2}$, the second
limit in \eqref{lim_P*} gives us the strong asymptotic of the sequence $(P_n)_{n \geq 0}$, We are done. \hfill $\Box$

Naturally, the asymptotic of the sequence of normalizing coefficients and of the orthonormal polynomials $p_{n,1},
p_{n,2}$ can also be given. We leave the details to the reader.

\subsection{Asymptotic of ML Hermite-Pad\'e polynomials}

In this subsection, as an easy consequence of the strong asymptotic of the polynomials $Q_{n,1}$ and $Q_{n,2}$, we
obtain the strong asymptotic  of $\mathcal{A}_{n,j}$, and $a_{n,j}$, $j=0,1$. Recall that $Q_{n,2}\equiv a_{n,2} \equiv
\mathcal{A}_{n,2}$.

Let $f$ be a function which has a constant sign on some interval $\Delta \subset \R$. We define
\[ \mbox{sg}_\Delta (f) :=
\left\{
\begin{array}{rr}
1, & f >0\,\,\, \mbox{on}\,\,\, \Delta, \\
-1, & f < 0\,\,\, \mbox{on}\,\,\, \Delta.
\end{array}
\right.
\]

\begin{corollary}
	\label{asym_forms}
	Let $(\sigma_1,\sigma_2) \in \mathcal{S}(\mathbf{\Delta})$  and $\mathcal{A}_{n,j}, j=0,1,$ is defined by
\eqref{JLS1}--\eqref{JLS2}. Then,
	\begin{equation}
		\label{asym_An1}
		\lim_n  \mbox{\rm sg}_{\Delta_2} (Q_{n,1}) \frac{\kappa_{n,2}^2\mathcal{A}_{n,1}(z)}{\left(\Phi_1/\Phi_2\right)^n(z)}
= \frac{\mathsf{G}_2(\infty)}{\mathsf{G}_1(\infty)} \frac{\mathsf{G}_1(z)}{\mathsf{G}_2(z)}  \frac{1}{\sqrt{(z-b_2)(z-a_2)}},
	\end{equation}
	and
	\begin{equation}
		\label{asym_An0}
		\lim_n \mbox{\rm sg}_{\Delta_1} \left(\frac{h_{n,1}}{Q_{n,2}}\right)
\frac{(\kappa_{n,1}\kappa_{n,2})^2\mathcal{A}_{n,0}(z)}{\Phi_1^{-n}(z)} = \frac{\mathsf{G}_1(z)}{\mathsf{G}_1(\infty)}
\frac{1}{\sqrt{(z-a_1)(z-b_1)}},
	\end{equation}	
	where the limits are uniform on compact subsets of $\overline{\C}\setminus(\Delta_1\cup\Delta_2)$ and
$\overline{\C}\setminus\Delta_1$, respectively.
\end{corollary}

\begin{proof}
	 Formula \eqref{hn1} can be rewritten as
	\begin{equation*}
	\mathcal{A}_{n,1}(z) =    \frac{Q_{n,1}(z)}{Q_{n,2}(z)}\int\frac{Q_{n,2}^2(x)}{z-x}\frac{\D\sigma_2(x)}{Q_{n,1}(x)} ,
	\end{equation*}
	where the equality holds in  $\Omega_2$. Then,
	\begin{equation}
		\label{asym_An1_1}
		\mbox{\rm sg}_{\Delta_2} (Q_{n,1}) \frac{\kappa_{n,2}^2\mathcal{A}_{n,1}(z)}{\left(\Phi_1/\Phi_2\right)^n(z)} =
\frac{Q_{n,1}(z)}{\Phi_1^n(z)}\frac{\Phi_2^n(z)}{Q_{n,2}(z)}\int\frac{q_{n,2}^2(x)}{z-x}\frac{\D\sigma_2(x)}{|Q_{n,1}(x)|},
\quad z\in\overline{\C}\setminus(\Delta_1\cup\Delta_2).
	\end{equation}
From  \eqref{int9} with $g_2(x) = (z-x)^{-1}$, we have
	\begin{equation}
		\label{asym_An1_3}
		\lim_n  \int\frac{q_{n,2}^2(x)}{z-x}\frac{\D\sigma_2(x)}{|Q_{n,1}(x)|}  =  \frac{1}{\sqrt{(z-b_2)(z-a_2)}},
	\end{equation}
	uniformly on compact subsets of $\Omega_2$. This, together with \eqref{limfund*}	 and \eqref{asym_An1_1}, gives
us
\eqref{asym_An1}.

Combining \eqref{int1} for $j=0$ with  \eqref{hn1} we get
	\begin{equation*}
		\mathcal{A}_{n,0}(z) = \int \frac{Q_{n,1}(x)}{z-x}\frac{\mathcal{H}_{n,1}(x)\D\sigma_1(x)}{Q_{n,2}(x)}.
	\end{equation*}
 By orthogonality, we have
	\begin{equation*}
		Q_{n,1}(z)\int \frac{Q_{n,1}(x)}{z-x}\frac{\mathcal{H}_{n,1}(x)\D\sigma_1(x)}{Q_{n,2}(x)} = \int
\frac{Q_{n,1}^2(x)}{z-x}\frac{\mathcal{H}_{n,1}(x)\D\sigma_1(x)}{Q_{n,2}(x)}.
	\end{equation*}	
Therefore,
	\begin{equation*}
		\mathcal{A}_{n,0}(z) = \frac{1}{Q_{n,1}(z)}\int
\frac{Q_{n,1}^2(x)}{z-x}\frac{\mathcal{H}_{n,1}(x)\D\sigma_1(x)}{Q_{n,2}(x)},
	\end{equation*}
	where the equality holds for $z\in\Omega_1$. So,
	\begin{equation}
		\label{asym_An0_1}
		\mbox{\rm sg}_{\Delta_1} \left(\frac{h_{n,1}}{Q_{n,2}}\right)
\frac{(\kappa_{n,1}\kappa_{n,2})^2\mathcal{A}_{n,0}(z)}{\Phi_1^{-n}(z)} = \frac{\Phi_1^n(z)}{Q_{n,1}(z)}\int
\frac{q_{n,1}^2(x)}{z-x}\frac{|h_{n,1}(x)|\D\sigma_1(x)}{|Q_{n,2}(x)|}.
	\end{equation}
From   \eqref{int10} with $g_1(x) = (z-x)^{-1}$, it follows that
	\begin{equation*}
		\label{asym_An0_3}
		\lim_n \int \frac{q_{n,1}^2(x)}{z-x}\frac{|h_{n,1}(x)|\D\sigma_1(x)}{|Q_{n,2}(x)|} = \frac{1}{\sqrt{(z-a_1)(z-b_1)}},
	\end{equation*}
uniformly on compact subsets of $\Omega_1$.
This formula combined with \eqref{asym_An0_1} and \eqref{limfund*} gives us  \eqref{asym_An0}. We have completed
the proof.
\end{proof}

Regarding the sign functions in \eqref{asym_An1} and \eqref{asym_An0} it is easy to deduce the following (see
\eqref{hn1})
\[ \mbox{\rm sg}_{\Delta_2} \left( Q_{n,1}\right)  =
\left\{
\begin{array} {rrr}
1, & n\,\, \mbox{is even}, & \mbox{}\\
1, & n\,\, \mbox{is odd}, & b_1 < a_2,\\
-1, & n\,\, \mbox{is odd}, & b_2 < a_1,
\end{array}
\right.
\quad
\mbox{\rm sg}_{\Delta_1} \left(\frac{h_{n,1}}{Q_{n,2}}\right)  =
\left\{
\begin{array} {rrr}
-1, & n\,\, \mbox{is even}, & b_1 < a_2,\\
1, & n\,\, \mbox{is even}, & b_2 < a_1,\\
1, & n\,\, \mbox{is odd}, & b_1 < a_2,\\
-1, & n\,\, \mbox{is odd}, & b_2 < a_1.
\end{array}
\right.
\]

Notice that
\[ \widehat{\sigma}_2(z) - \frac{a_{n,1}(z)}{a_{n,2}(z)} = \frac{\mathcal{A}_{n,1}(z)}{Q_{n,2}(z)} =
\frac{Q_{n,1}(z)}{q^2_{n,2}(z)}\int\frac{q_{n,2}^2(x)}{z-x}\frac{\D\sigma_2(x)}{Q_{n,1}(x)}.
\]
Therefore, using \eqref{limfund*}, \eqref{limfund**}, and \eqref{asym_An1_3}, we obtain
\begin{equation} \label{cero} \lim_n \left| \frac{C_2^2 \Phi_2^2(z)}{\Phi_1(z)}\right|^n\left| \widehat{\sigma}_2(z) -
\frac{a_{n,1}(z)}{a_{n,2}(z)} \right|
= \frac{2\pi |\mathsf{G}_1(z)|}{\mathsf{G}_1(\infty)|\mathsf{G}^2_2(z)|\sqrt{|z-a_2||z-b_2|}},
\end{equation}
uniformly on compact subsets of $\overline{\C} \setminus (\Delta_1 \cup \Delta_2)$.
The definition of $\Phi_1,\Phi_2$, and $C_2$ imply
\[  \left| \frac{C_2^2 \Phi_2^2(z)}{\Phi_1(z)}\right| = \exp\left(-(2V_{\lambda_2}(z) - V_{\lambda_1}(z) - 2\gamma_2)\right).
\]
The second equilibrium equation in \eqref{vector_equil} and the maximum principle for subharmonic function entail
\[ 2V_{\lambda_2}(z) - V_{\lambda_1}(z) - 2\gamma_2 < 0, \qquad z \in \Omega_2.
\]
Consequently, \eqref{cero} gives a precise description of the rate with which $({a_{n,1}}/{a_{n,2}})_{n\geq 0}$ converges
to $\widehat{\sigma}_2$. Exactly the same formula can be obtained substituting in \eqref{cero} $\widehat{\sigma}_2 -
a_{n,1}/a_{n,2}$ by $\widehat{s}_{2,1} - a_{n,0}/a_{n,2}$.  However, we will not dwell into this because it requires the
introduction of   new transformations which drive us off track.

From \cite[Th. 1.6]{LMS}, we know that, for $j=0,1$
\begin{equation*}
	\lim_n \frac{a_{n,j}(z)}{a_{n,2}(z)}=\hat{s}_{2,j+1}(z),
\end{equation*}	
where the limit is uniform on compact subsets of $\Omega_2$.

\begin{corollary}
	\label{asym_pol_anj}
	Let $(\sigma_1,\sigma_2) \in \mathcal{S}(\mathbf{\Delta})$. Then, for $j=0,1$
	\begin{equation*}
		\label{lim_pol_anj}
		\lim_n \frac{a_{n,j}(z)}{\Phi_2^n(z)}=\frac{\mathsf{G}_2(z)}{\mathsf{G}_2(\infty)}\hat{s}_{2,j+1}(z),
	\end{equation*}
	where the limit is uniform on compact subsets of $\Omega_2$.
\end{corollary}

\begin{proof}
	Since $a_{n,2}\equiv Q_{n,2}$, we have
	\begin{equation*}
		\label{pol_anj_1}
		\lim_n \frac{a_{n,j}(z)}{\Phi_2^n(z)}\frac{\Phi_2^n(z)}{Q_{n,2}(z)}=\hat{s}_{2,j+1}(z).
	\end{equation*}	
	Taking into account \eqref{limfund**} the proof readily follows.
\end{proof}	

Results analogous to Corollaries \ref{asym_forms} and \ref{asym_pol_anj} for the forms $\mathcal{B}_{n,0},
\mathcal{B}_{n,1},$ and the polynomials $b_{n,0}, b_{n,1},$ follow immediately considering the Nikishin system
$\mathcal{N}(\sigma_2,\sigma_1)$. The details are left to the reader.

\subsection{A different expression for the  functions $\Phi_1, \Phi_2$ and the constants $C_1,C_2$}

In \cite[Theorem 4.2]{ULS} the ratio asymptotic of general ML Hermite-Pad\'e polynomials was given. The limit was expressed in terms of the branches of a conformal map of a certain Riemann surface. Since strong asymptotic implies ratio asymptotic, we can use that result to  interpret the comparison functions $\Phi_1,  \Phi_2$ and the constants $C_1,C_2$ in a different way. (Which coincides with the form in which they were defined in \cite{sasha}.)

We introduce the Riemann surface which is relevant in our case of two measures. Let $\mathcal{R}$ denote the compact Riemann surface
\[
\mathcal{R}=\overline{\bigcup_{k=0}^{2}\mathcal{R}_{k}}
\]
formed by $3$ consecutively ``glued" copies of the extended complex plane
\[
\mathcal{R}_{0}:=\overline{\mathbb{C}}\setminus\Delta_{1},\qquad
\mathcal{R}_{1}:=\overline{\mathbb{C}}\setminus(\Delta_{1}\cup\Delta_{2}),\quad
\mathcal{R}_{2}:=\overline{\mathbb{C}}\setminus\Delta_{2}.
\]
The upper and lower banks of the slits of two neighboring sheets are identified.

Let $\pi: \mathcal{R} \longrightarrow \overline{\mathbb{C}}$ be the canonical projection from $\mathcal{R}$ to $\overline{\mathbb{C}}$ and denote by $z^{(k)}$ the point on  $\mathcal{R}_k$ verifying $\pi(z^{(k)}) = z, z \in \overline{\mathbb{C}}$. Let $\phi :\mathcal{R}\longrightarrow\overline{\mathbb{C}}$ denote a conformal mapping whose divisor consists of one simple zero at $\infty^{(0)}\in\mathcal{R}_{0}$ and one simple pole at $\infty^{(2)}\in\mathcal{R}_{2}$. This mapping exists and is uniquely determined up to a multiplicative constant. Denote the branches of $\phi$ by
\[ \phi_k (z) := \phi (z^{(k)}), \qquad k= 0,1,2, \qquad z^{(k)} \in \mathcal{R}_k.\]
From the properties of $\phi $, we have
\begin{equation}\label{divisorcond}
\phi_0 (z)=c_{1}/z+\mathcal{O}(1/z^{2}),\qquad \phi_2(z)=c_{2}\,z+\mathcal{O}(1),\qquad z\rightarrow\infty,
\end{equation}
where $c_{1}$, $c_{2}$ are non-zero constants.

Let $x \in \Delta_k, k=1,2$. We write $z \to x_+$ when $z \in \mathbb{C}$ approaches $x$ from above the real line. Analogously, $z \to x_- $ means that $z$ approaches $x$ from below the real line. Let us define
\[ \phi_k (x_+) := \lim_{z \to x_+} \phi_k (z) = \lim_{z \to x_+} \phi (z^{(k)})
\]
and
\[ \phi_k (x_-) := \lim_{z \to x_-} \phi_k (z) = \lim_{z \to x_-} \phi (z^{(k)}).
\]
Except when $x$ is an end point of $\Delta_k$, these limits are different due to the fact that $\lim_{z\to x_+} z^{(k)} \neq \lim_{z\to x_-} z^{(k)}$ on $\mathcal{R}$. However, due to the identification made of the points on the slits it is easy to verify that
\begin{equation}\label{identif}
\phi_k (x_+) = \phi_{k+1} (x_-), \qquad \phi_k (x_-) = \phi_{k+1} (x_+), \qquad k=0,1,
\end{equation}
because
\[ \lim_{z \to x_+} z^{(k)} = \lim_{z \to x_-} z^{(k+1)}, \qquad \lim_{z \to x_-} z^{(k)} = \lim_{z \to x_+} z^{(k+1)}.
\]

Taking account of the way in which the functions $\phi_k $ were extended to $\Delta_k$ and \eqref{identif} it follows that $ \prod_{k=0}^{2}\phi_{k} $ is a single-valued analytic function on $\overline{\mathbb{C}}$ without singularities; therefore, it is constant. We normalize $\phi $ so that
\[ \prod_{k=0}^{2}\phi_{k} = c, \qquad |c| = 1, \qquad c_{1} > 0.
\]
Let us show that with this normalization $c$ is  $+1$.

Indeed, for a point $z^{(k)} \in \mathcal{R}_k$ on the Riemann surface we define its conjugate $\overline{z^{(k)}} := \overline{z}^{(k)}$. For a points $z^{(k)}$ on the upper bank of the slit $\Delta_k$ the conjugate is the one corresponding to the lower bank. Now, we define $ {\phi}^* : \mathcal{R} \longrightarrow \overline{\mathbb{C}}$ as follows  ${\phi}^*(\zeta):= \overline{\phi (\overline{\zeta})}$. It is easy to verify that $ {\phi}^*$ is a conformal mapping of $\mathcal{R}$ onto $\overline{\mathbb{C}}$ with the same divisor as $\phi $. Therefore, there exists a constant $\kappa$ such that $ {\phi}^* = \kappa\phi$. The corresponding branches satisfy the relations
\[ {\phi}^*_k (z) = \overline{\phi_k(\overline{z})} = \kappa {\phi}_k (z), \qquad k=0,1,2.
\]
Comparing the Laurent expansions at $\infty$ of $\overline{\phi_0 (\overline{z})}$ and $\kappa {\phi}_0 (z)$, using the fact that $c_{1} >0$, it follows that $\kappa = 1$. Then
\[   {\phi}_k (z) = \overline{\phi_k (\overline{z})}, \qquad k=0,1,2.
\]
This in turn implies that for each $k=0,1,2$ all the coefficients, in particular the leading one, of the Laurent expansion at infinity of $ {\phi}_k $ are real numbers. Obviously, $c$ is the product of these leading coefficients. Since they are real numbers $c$ is real and since it is of module $1$, it has to be either $1$ or $-1$. Analyzing the Laurent expansion of the branches at $\infty$ one easily concludes that indeed $c=1$. So, we can assume in the following that
\begin{equation}\label{normconfmap}
\prod_{k=0}^{2}\phi_{k} \equiv   1,\qquad c_{1}>0.
\end{equation}
It is easy to see that conditions \eqref{divisorcond} and \eqref{normconfmap} determine $\phi $ uniquely.

The question of finding explicit expressions for conformal representations of three sheeted Riemann surfaces of genus zero depending on the values of the end points of the intervals $\Delta_1,\Delta_2$ was considered in \cite{LPRY}. This problem is not solvable in closed form. \cite[Theorem 3.1]{LPRY} gives an expression in terms of the solution of a system of two nonlinear equations of higher order. Already the simpler case of two intervals of equal length requires the solution of a bicuartic equation \cite[Theorem 3.3]{LPRY}. A numerical method for solving the system of equations is given in \cite[Theorem 6.1]{LPRY}.

In \cite[Lemma 4.2]{GSI} the authors proved the following result.
\begin{lemma}
	\label{uniqueness}
	Their exists a unique pair of functions  $(F_1,F_2)$ such that for $k=1,2$
 \begin{enumerate}
 	\item $F_k, 1/F_k\in\mathcal{H}(\C\setminus\Delta_k)$,
 	\item $F_k'(\infty)>0$,
 	\item $\frac{|F_k(x)|^2}{|F_{k-1}(x)F_{k+1}(x)|}=1$, $x\in\Delta_k$,
 \end{enumerate}
$(F_0\equiv F_3 \equiv 1)$. The functions may be expressed by the formulas
\begin{equation*}
F_k :=\prod_{\nu=k}^{2}\phi_\nu, \qquad k=1,2.
\end{equation*}
\end{lemma}
The boundary conditions for the functions $F_k$, $k=1,2$ are
\begin{align*}
|F_1(x)|^2 =& |F_2(x)|,\qquad x\in\Delta_1,\\
|F_2(x)|^2 =& |F_1(x)|,\qquad x\in\Delta_2.
\end{align*}
Compare with \eqref{vector_equil} after taking logarithm.

From \cite[Theorem 4.2, Corollary 4.3]{ULS} we know that for $k=1,2,$
\[ \lim_{n\to \infty} \frac{Q_{n+1,k}(z)}{Q_{n,k}(z)} = \frac{F_k(z)}{F'_k(\infty)},
\]
uniformly on compact subsets of $\Omega_k$ and
\[\lim_{n \to \infty} \frac{\kappa_{n+1,k}}{\kappa_{n,k}} = \frac{F_k'(\infty)}{\sqrt{F_{k-1}'(\infty)F_{k+1}'(\infty)}},
\]
where by definition we take $F_0'(\infty) = F_3'(\infty)= 1$. On the other hand, \eqref{limfund*} and \eqref{conductor*} imply that
\[ \lim_{n\to \infty} \frac{Q_{n+1,k}(z)}{\Phi_k^{n+1}(z)} \frac{\Phi_k^{n}(z)}{Q_{n,k}(z)} = \frac{1}{\Phi_k(z)} \lim_{n\to
\infty} \frac{Q_{n+1,k}(z)}{Q_{n,k}(z)}  = 1,
\]
uniformly on compact subsets of $\Omega_k$ and
\[ \lim_{n\to \infty} \frac{\kappa_{n+1,k}}{C_k^{n+1}} \frac{C_k^{n}}{\kappa_{n,k}} = \frac{1}{C_k} \lim_{n\to \infty}
\frac{\kappa_{n+1,k}}{\kappa_{n,k}}  = 1.
\]
 Consequently
\begin{equation} \label{otro} \Phi_k(z) \equiv \frac{F_k(z)}{F'_k(\infty)}, \qquad C_k =
\frac{F_k'(\infty)}{\sqrt{F_{k-1}'(\infty)F_{k+1}'(\infty)}}, \qquad k=1,2.
\end{equation}

\end{document}